\documentclass[pdflatex,sn-mathphys-num]{sn-jnl}

\usepackage{graphicx}%
\usepackage{multirow}%
\usepackage{amssymb}
\usepackage{amsmath}
\usepackage{amsfonts}
\usepackage{mathrsfs}
\usepackage{mathtools}
\usepackage[title]{appendix}%
\usepackage{xcolor}%
\usepackage{textcomp}%
\usepackage{centernot}
\usepackage{manyfoot}%
\usepackage{booktabs}%
\usepackage{algorithm}%
\usepackage{algorithmicx}%
\usepackage{algpseudocode}%
\usepackage{listings}%
\usepackage{enumitem}

\usepackage{tikz}
\usetikzlibrary{backgrounds,arrows.meta,positioning,fit,calc,decorations.pathreplacing,shapes.multipart,shapes.geometric}
\usepackage{cleveref}

\theoremstyle{thmstyleone}
\newtheorem{theorem}{Theorem}[section]
\newtheorem{proposition}[theorem]{Proposition}

\newtheorem{corollary}[theorem]{Corollary}
\theoremstyle{definition}

\theoremstyle{remark}
\newtheorem{remark}[theorem]{Remark}

\begin{document}

\title{The Geometry of Quasi-Cycles: How Stoichiometric Covariance Alters Pre-Bifurcation Signatures}

\author*[1]{\fnm{Louis Shuo} \sur{Wang}}\email{wang.s41@northeastern.edu}
\equalcont{These authors contributed equally to this work as co-first authors.}

\author[2,3]{\fnm{Jiguang} \sur{Yu}}\email{jyu678@bu.edu}
\equalcont{These authors contributed equally to this work as co-first authors.}

\author[4]{\fnm{Ye} \sur{Liang}}\email{ye-liang@uiowa.edu}

\author[5]{\fnm{Jilin} \sur{Zhang}}\email{jilin.zhang25@imperial.ac.uk}

\affil[1]{\orgdiv{Department of Mathematics},
  \orgname{Northeastern University},
  \orgaddress{\city{Boston}, \postcode{02115}, \state{MA}, \country{USA}}}

\affil[2]{\orgdiv{College of Engineering},
  \orgname{Boston University},
  \orgaddress{\city{Boston}, \postcode{02215}, \state{MA}, \country{USA}}}

\affil[3]{\orgdiv{GSS Fellow of Institute for Global Sustainability},
  \orgname{Boston University},
  \orgaddress{\city{Boston}, \postcode{02215}, \state{MA}, \country{USA}}}

\affil[4]{\orgdiv{Department of Industrial and Systems Engineering},
  \orgname{The University of Iowa},
  \orgaddress{\city{Iowa City}, \postcode{52242}, \state{IA},\country{USA}}}

\affil[5]{\orgdiv{Department of Mathematics},
  \orgname{Imperial College London},
  \orgaddress{\city{London}, \postcode{SW7 2AZ}, \country{UK}}}
  
\abstract{
Environmental enrichment can destabilize predator–prey coexistence through a Hopf bifurcation, yet real ecosystems are finite and intrinsically stochastic. We investigate how mechanistically derived demographic noise shapes near-Hopf dynamics in the Rosenzweig–MacArthur model by systematically comparing two diffusion closures that share identical deterministic drift but differ solely in predation-induced covariance structure. Starting from a continuous-time Markov chain description, we derive a full-covariance stochastic differential equation whose diffusion tensor inherits stoichiometric coupling, generating a negative prey–predator cross-covariance. This model is contrasted with a drift-matched diagonal-noise comparator. Using linear noise approximation, Lyapunov analysis, and matrix-valued power spectral density formulations, we propagate local covariance structure through the entire diagnostic chain, including stochastic sensitivity ellipses and a dimensionless noisy-precursor indicator. 
The results highlight that drift equivalence does not imply covariance equivalence and show how event-level noise geometry influences macroscopic behavior in nonlinear ecological systems. This work integrates bifurcation theory and stochastic analysis to advance multi-scale modeling of complex interacting systems.}


\maketitle

\section{Introduction}
\label{sec:introduction}

\begin{figure}[htbp]
\centering
\resizebox{\linewidth}{!}{%
\begin{tikzpicture}[
  x=1mm, y=1mm,
  font=\scriptsize,
  every node/.style={text=black},
  >=Latex,
  box/.style={draw, rounded corners=2pt, thick, align=center, inner sep=4pt, fill=white},
  sbox/.style={draw, rounded corners=2pt, thick, align=center, inner sep=3.5pt, fill=white},
  pbox/.style={draw, rounded corners=2pt, thick, align=left, inner sep=4pt, fill=white},
  arrow/.style={->, thick, shorten <=2pt, shorten >=2pt, rounded corners=2pt},
  dashedarrow/.style={->, thick, dashed, shorten <=2pt, shorten >=2pt, rounded corners=2pt},
  lab/.style={font=\scriptsize, align=center},
  emph/.style={font=\scriptsize\bfseries, align=center}
]

\coordinate (L) at (0,0);
\coordinate (C) at (60,0);
\coordinate (R) at (120,0);

\def\yTop{0}
\def\ySto{-22}
\def\yComp{-48}
\def\yLNA{-74}
\def\yDiag{-100}
\def\yBottom{-128}
\def\yMsg{-144}

\pgfmathsetmacro{\dropToBottom}{\yBottom-\yTop+10}

\node[box, minimum width=45mm] (det) at ($(L)+(0,\yTop)$) {
Deterministic R--M backbone\\
$\dot x=b(x)$\\
Enrichment $k\uparrow \Rightarrow k\to k_H$
};

\node[box, minimum width=46mm] (ctmc) at ($(R)+(0,\yTop)$) {
Finite populations\\
CTMC/CME description\\
birth / death / predation
};

\node[sbox, minimum width=70mm] (stoich) at ($(C)+(0,\ySto)$) {
Stoichiometry \& diffusion limit\\
$a(x)=\frac1\Omega S\,\mathrm{diag}(f(x))S^\top$
};

\node[pbox, minimum width=52mm] (full) at ($(L)+(0,\yComp)$) {
	Full covariance\\
Coupled predation event\\
$q_{12}<0$ in $D_*=a(K_3)$
};

\node[pbox, minimum width=52mm] (diag) at ($(R)+(0,\yComp)$) {
	Diagonal comparator\\
Independent channels\\
$q_{12}=0$ in $D_*$
};

\node[box, minimum width=108mm] (lna) at ($(C)+(0,\yLNA)$) {
LNA / OU near coexistence $K_3$ ($k<k_H$)\\
$dy_t=J\,y_t\,dt+B\,dW_t,\quad D_*=BB^\top$
};

\node[sbox] (lyap) at ($(L)+(0,\yDiag)$) {
Lyapunov equation\\
$J W+WJ^\top+D_*=0$
};

\node[sbox] (psd) at ($(C)+(0,\yDiag)$) {
Matrix PSD\\
$S(\omega)$ peak
};

\node[sbox] (ssf) at ($(R)+(0,\yDiag)$) {
SSF ellipse\\
$z^\top W^{-1}z\le \chi^2_2(p)$
};

\node[box, minimum width=46mm] (open) at ($(L)+(0,\yBottom)$) {
Open-domain viewpoint\\
quasi-cycles / precursors
};

\node[box, minimum width=46mm] (ssa) at ($(C)+(0,\yBottom)$) {
SSA (Gillespie)\\
Exact CTMC benchmark
};

\node[box, minimum width=46mm] (abs) at ($(R)+(0,\yBottom)$) {
Absorbed viewpoint\\
extinction statistics
};

\node[emph] (msg) at ($(C)+(0,\yMsg)$) {
Mechanistic noise coupling ($q_{12}<0$) reshapes near-Hopf fluctuation geometry and spectral amplification.
};

\draw[arrow] (det) -- node[lab, above]{same $b(x)$} (ctmc);

\draw[dashedarrow] (ctmc.south) -- ++(0,-8) -| node[lab, right]{diffusion limit} (stoich.east);

\draw[arrow] (stoich.south west) -- (full.north east);
\draw[arrow] (stoich.south east) -- (diag.north west);

\draw[arrow] (full.south) -- ++(0,-6) -| (lna.west);
\draw[arrow] (diag.south) -- ++(0,-6) -| (lna.east);

\draw[arrow] (lna.south) -- (psd.north);
\draw[arrow] (lna.south west) -- (lyap.north east);
\draw[arrow] (lna.south east) -- (ssf.north west);

\draw[dashedarrow] (lyap.south) -- (open.north);
\draw[dashedarrow] (ssf.south) -- (abs.north);

\draw[arrow] (ssa.north) -- (psd.south);

\begin{pgfonlayer}{background}
\node[draw, very thick, rounded corners=3pt, fill=black!3,
  fit=(det)(ctmc)(stoich)(full)(diag)(lna)(lyap)(psd)(ssf)(open)(abs)(ssa)(msg),
  inner sep=9pt] {};
\end{pgfonlayer}

\end{tikzpicture}%
}
\caption{\textbf{Mechanism overview.}
Environmental enrichment drives the deterministic backbone toward Hopf.
Stoichiometric coupling determines the diffusion covariance.
Two drift-matched closures differing only in off-diagonal covariance propagate distinct fluctuation geometry through LNA diagnostics and extinction-relevant metrics.}
\label{fig:mechanism_overview}
\end{figure}
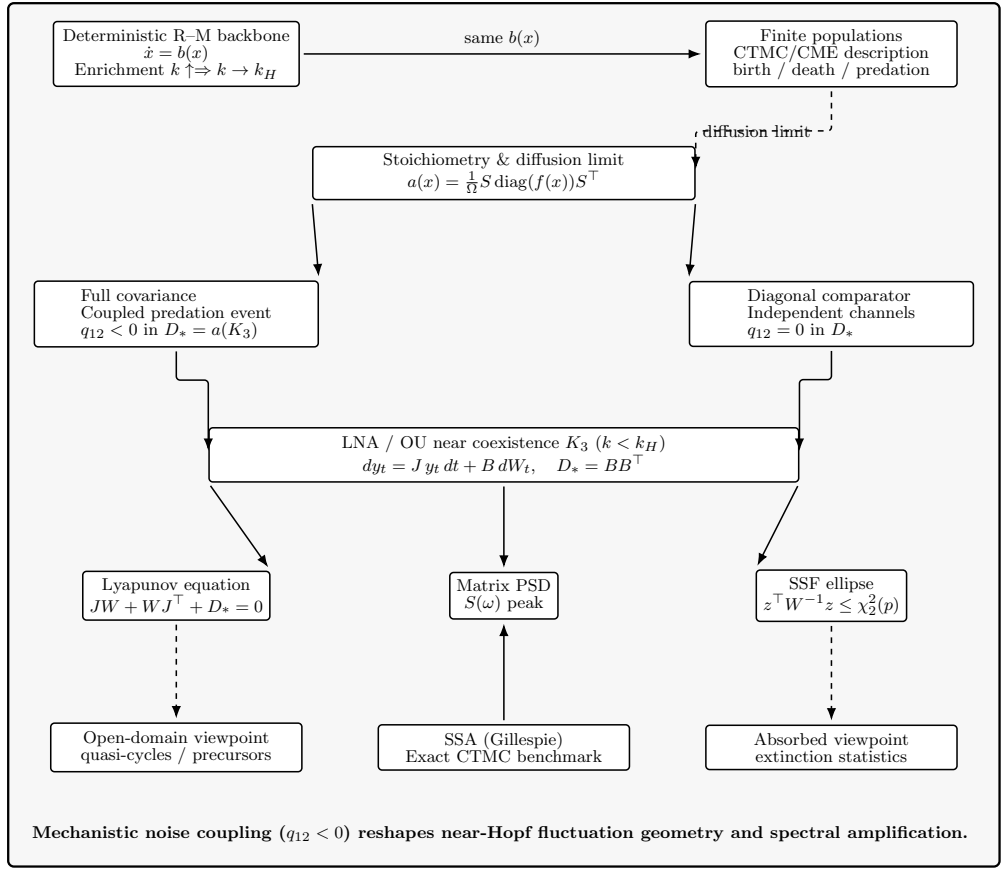

Understanding the mechanisms by which ecosystems transition from stable coexistence to sustained oscillatory behavior is a central problem in theoretical ecology. 
In deterministic predator--prey models of the Rosenzweig--MacArthur (R--M) type \citep{rosenzweig_graphical_1963,macarthur_species_1970,rosenzweig_paradox_1971}, 
progressive environmental enrichment constitutes an increase in prey carrying capacity.
This increase drives the coexistence equilibrium through a supercritical Hopf bifurcation and generates large-amplitude limit cycles that can substantially elevate extinction risk \citep{gilpin_enriched_1972,may_limit_1972,riebesell_paradox_1974}. 
This paradox of enrichment has remained one of the most discussed phenomena in population ecology for over five decades \citep{rosenzweig_paradox_1971,abrams_invulnerable_1996,kirk_enrichment_1998,roy_stability_2007}.

Real populations, however, are finite. 
Demographic stochasticity is the inherent randomness arising from discrete birth, death, and predation events in populations of finite size.
This randomness ensures that the approach to a Hopf bifurcation is accompanied by growing fluctuations even while the deterministic equilibrium remains locally stable \citep{mckane_predator-prey_2005,boland_how_2008,pineda-krch_tale_2007}. 
These noise-driven, quasi-periodic fluctuations are commonly termed quasi-cycles or stochastic amplification.
They produce coherent oscillatory signatures in both time series and power spectra that can serve as early-warning signals of an impending bifurcation \citep{kuehn_mathematical_2011,scheffer_early-warning_2009,scheffer_anticipating_2012,dakos_methods_2012}. 
Detecting and interpreting such noisy precursors has practical relevance for ecosystem management, where interventions must typically be made before the deterministic bifurcation point is actually crossed \citep{boettiger_quantifying_2012,boettiger_patterns_2013,lenton_tipping_2008}.

The linear noise approximation (LNA) provides the standard analytical bridge between the stochastic birth--death process and its continuous diffusion limit \citep{van_kampen_stochastic_2011,kurtz_solutions_1970,kurtz_limit_1971,gardiner_stochastic_2009}. 
Within the LNA framework, fluctuations near a stable equilibrium are described by an Ornstein--Uhlenbeck (OU) process. 
The stationary covariance matrix of this process satisfies a continuous Lyapunov equation, and its spectral properties are encoded in a matrix-valued power spectral density (PSD) \citep{gardiner_stochastic_2009,risken_fokker-planck_1996}. 
As the Hopf threshold is approached, the resolvent of the linearized dynamics develops a resonance near the frequency of the incipient limit cycle, producing a characteristic spectral peak in the PSD that sharpens and grows in amplitude \citep{mckane_predator-prey_2005,boland_how_2008,lugo_quasicycles_2008,black_stochastic_2012}.

A key modeling choice concerns the covariance structure of the demographic noise.
This choice is often made implicit and forms the central focus of this paper.
When a stochastic population model is derived from a continuous-time Markov chain (CTMC) or chemical master equation (CME) description of discrete events, the diffusion covariance in the resulting stochastic differential equation (SDE) inherits the stoichiometric coupling of those events \citep{kurtz_solutions_1970,gillespie_exact_1977,gillespie_chemical_2000,ethier_markov_1986,anderson_stochastic_2015}. 
In predator--prey systems, predation is mechanistically a coupled event: a single encounter simultaneously removes a prey individual and, with some probability, produces a predator offspring. 
This same-event coupling generates a negative prey--predator cross-covariance in the diffusion tensor \citep{wang_algebraicspectral_2026}. 
Many stochastic analyses of ecological models, including influential studies of quasi-cycles, employ diagonal (uncorrelated) or phenomenologically specified noise \citep{boland_how_2008,alonso_stochastic_2007,wang_analysis_2025,chatterjee_predatorprey_2024,pal_nonequilibrium_2025}. 
They effectively treat prey loss and predator gain as statistically independent channels.

Whether this simplification matters for near-Hopf diagnostics has not been systematically examined. 
This paper addresses that gap by developing and comparing two diffusion closures that share an identical R--M deterministic drift but differ precisely in their predation-induced covariance structure.
We contrast a full-covariance model retaining the mechanistically inherited negative cross-covariance with a drift-matched diagonal-noise comparator treating predation-related prey loss and predator gain as independent channels. 
The comparison is ``fair'' in the sense that the sole modeling difference resides in the off-diagonal diffusion covariance; deterministic dynamics, marginal noise intensities, and parameter values are held identical.

We propagate the full non-diagonal local diffusion covariance through the entire LNA diagnostic chain, which includes stationary covariance, matrix PSD, and stochastic sensitivity function (SSF) confidence ellipses. 
This analysis demonstrates that the covariance structure of demographic noise affects not only the magnitude but also the geometry of fluctuations near a Hopf bifurcation.

Figure~\ref{fig:mechanism_overview} provides a mechanistic overview of our work. It is organized as follows. 
Section~\ref{sec:deterministic} reviews the deterministic R--M model, its equilibria, and the enrichment-driven Hopf bifurcation.  Section~\ref{sec:stochastic_model} defines the stochastic diffusion models used throughout the paper, with emphasis on the full-covariance and diagonal-noise closures and the structural sign of the predation-induced cross-covariance.  Section~\ref{sec:formulations} discusses the open-domain and absorbed boundary viewpoints for interpretation. Section~\ref{sec:quasi_cycles} develops the main analytical framework: the LNA/OU reduction, Lyapunov covariance, matrix PSD, stochastic sensitivity ellipses, and the noisy-precursor indicator. 
The appendices provide detailed derivations of the PSD (Appendix~\ref{app:LNA_PSD}) alongside the $2\times 2$ Lyapunov solution and SSF formulas (Appendix~\ref{app:SSF_Lyapunov_2x2}).
They also contain explicit covariance closures (Appendix~\ref{app:covariance_closures}), 
the Jacobian at the coexistence equilibrium (Appendix~\ref{app:Jacobian_K3}), 
and technical proofs for the covariance-structure results (Appendix~\ref{app:covariance_structure_proofs}).

\section{Literature review}
\label{sec:literature_review}

\subsection{The paradox of enrichment and Hopf bifurcations in predator--prey models}

The theoretical prediction that nutrient enrichment can destabilize predator--prey coexistence dates to the seminal work of Rosenzweig \citep{rosenzweig_paradox_1971}. 
He showed that increasing carrying capacity pushes the interior equilibrium through a Hopf bifurcation, triggering limit-cycle oscillations of growing amplitude. 
This observation, termed the paradox of enrichment, was formalized within the R--M model \citep{hauzy_confronting_2013,jansen_regulation_1995}, whose type-II functional response provides the minimal mechanistic skeleton needed to produce the bifurcation.  Subsequent analytical studies established the supercritical character of the Hopf bifurcation \citep{kuznetsov_elements_2023,lin_bifurcation_2023}, while Bazykin et al. \citep{bazykin_nonlinear_1998} explored additional codimension-two degeneracies and global bifurcation structures in enriched predator--prey systems. Related stability-and-threshold analyses for feedback-regulated structured population models (including nonlocal operators and reversible transitions) further illustrate how coupling architecture determines global well-posedness and bifurcation structure; see \citep{liang_global_2025}.

\subsection{Demographic stochasticity, diffusion approximations, and the CME framework}

Stochastic formulations of population dynamics have a long history, beginning with the birth--death process models of Feller \citep{feller_grundlagen_1939} and Bartlett \citep{bartlett_stochastic_1970}, and the systematic development of the CME framework for reaction networks \citep{mcquarrie_stochastic_1967,gillespie_markov_1992}. 
In the ecological setting, demographic stochasticity is the randomness intrinsic to integer-valued population processes at finite abundance.
It serves as the primary source of internal noise. 
Representing the population dynamics as a CTMC on a discrete state space provides a rigorous probabilistic description using transition rates derived from mechanistic birth, death, and interaction events.
The moments and distributions of this model can then be related to macroscopic dynamics through systematic approximation procedures
\citep{kurtz_solutions_1970,kurtz_limit_1971,ethier_markov_1986,anderson_stochastic_2015,anderson_continuous_2011}.

The diffusion approximation was formalized by Kurtz's limit theorems \citep{kurtz_solutions_1970,kurtz_limit_1971,kurtz_strong_1978}.
This approximation replaces the discrete CTMC with a continuous-path It\^o SDE.
The drift of this SDE coincides with the deterministic rate equations.
Furthermore, the diffusion covariance is determined by the stoichiometry and intensities of the underlying events. 
This approximation becomes increasingly accurate as the system-size parameter $\Omega$ grows. 
The resulting density-level SDE provides the starting point for the LNA and for the stochastic sensitivity analyses developed in this paper.
van Kampen's system-size expansion \citep{van_kampen_stochastic_2011} and related $\Omega$-expansions \citep{gardiner_stochastic_2009,grima_effective_2010,grima_construction_2011,thomas_rigorous_2012} provide formal perturbative frameworks within which the LNA emerges as the leading-order Gaussian approximation.

For exact stochastic simulation, Gillespie's SSA \citep{gillespie_exact_1977,gillespie_general_1976,wang_analysis_2025-1} and its $\tau$-leaping variants \citep{gillespie_approximate_2001,cao_efficient_2006} generate statistically exact realizations of the CTMC. These serve as the ground-truth benchmark against which analytical approximations such as the LNA can be validated.

\subsection{Quasi-cycles, stochastic amplification, and the LNA}

Bartlett \citep{bartlett_stochastic_1970} and Nisbet and Gurney \citep{nisbet_modelling_1982} identified noise-induced quasi-periodic oscillations near a stable focus in an ecological context.
This phenomenon, now commonly called a quasi-cycle, was later given a modern spectral characterization through the LNA by McKane and Newman \citep{mckane_predator-prey_2005} and Boland et al. \citep{boland_how_2008}.
The linearized dynamics near a stable spiral point act as a narrow-band filter.
These dynamics selectively amplify white demographic noise near the damped oscillation frequency, producing a coherent spectral peak that diverges in height as the Hopf bifurcation is approached.
Complementary spectral threshold frameworks for coupled systems have been developed via Laplacian eigenmode reduction in hybrid PDE--ODE microenvironments, yielding explicit trace/determinant criteria for unstable modes; see \citep{yu_mode-wise_2026}.

This amplification mechanism has been studied extensively in both ecological and epidemiological settings. 
In ecology, quasi-cycles have been analyzed in predator--prey \citep{mckane_predator-prey_2005,pineda-krch_tale_2007,black_stochastic_2012,barraquand_moving_2017} and spatially extended metapopulation frameworks \citep{lugo_quasicycles_2008,mckane_stochastic_2014,butler_robust_2009,biancalani_stochastic_2010,biancalani_stochastic_2011}.
In epidemiology, the closely analogous phenomenon of stochastic resonance in SIR-type models has been studied by Alonso et al. \citep{alonso_stochastic_2007}, Rozhnova and Nunes \citep{rozhnova_stochastic_2010}, and Black and McKane \citep{black_stochastic_2010}, among others. 
A unifying perspective on noise amplification near bifurcations in low-dimensional stochastic systems has been provided by several reviews \citep{mckane_stochastic_2014,boland_limit_2009,allen_introduction_2010}.

The LNA framework yields closed-form expressions for both the stationary covariance (via the Lyapunov equation) and the PSD (via the transfer-function) of the fluctuation process \citep{van_kampen_stochastic_2011,gardiner_stochastic_2009,risken_fokker-planck_1996}.  In two-dimensional predator--prey models, these quantities can be computed explicitly as functions of the Jacobian $J$ at the coexistence equilibrium and the local diffusion covariance $D_*$, as we exploit in Section~\ref{sec:quasi_cycles}. 
The dependence of the PSD on $J$ and the distance to the Hopf threshold has been thoroughly studied.
In contrast, the role of the noise covariance $D_*$, particularly its off-diagonal structure, has received comparatively less attention in the ecological quasi-cycle literature.
The present paper aims to fill this research gap. 

\subsection{Stochastic sensitivity and confidence-ellipse geometry}

The SSF approach is based on the quasipotential theory \citep{freidlin_random_1998,dembo_large_2010}. 
Milshtein and Ryashko \citep{milshtein_first_1995}  developed this framework, which was subsequently extended by Bashkirtseva and Ryashko \citep{bashkirtseva_stochastic_2004,bashkirtseva_sensitivity_2000,bashkirtseva_stochastic_2013,bashkirtseva_sensitivity_2011}.
The method provides a geometric interpretation of noise-induced variability around deterministic attractors.
For stable equilibria, the SSF reduces to the stationary covariance of the LNA, and the associated confidence ellipses quantify the size, shape, and orientation of the fluctuation cloud \citep{bashkirtseva_sensitivity_2011,ryashko_sensitivity_2018}.

In the context of predator--prey and related models, stochastic sensitivity methods have been used to study noise-induced transitions between coexisting attractors \citep{alexandrov_noise-induced_2018} and extinction risk amplification near bifurcation thresholds \citep{bashkirtseva_noise-induced_2016}.
Furthermore, these methods have explored noise-induced transfer to chaotic regimes \citep{bashkirtseva_sensitivity_2005} and the geometry of fluctuations around limit cycles \citep{bashkirtseva_sensitivity_2000,bashkirtseva_stochastic_2004}.  
The connection between ellipse orientation and mechanistic noise coupling, however, has not been explored systematically in the ecological literature. 
Most applications employ diagonal or scalar-multiplied noise, so that the SSF ellipse orientation is governed entirely by the Jacobian rather than reflecting event-level stoichiometric structure. 

\subsection{Early-warning signals and noisy precursors of bifurcations}

The idea that statistical signatures in stochastic time series can provide advance warning of approaching tipping points has generated intense interest. 
This research area has expanded significantly over the past two decades \citep{scheffer_early-warning_2009,scheffer_anticipating_2012,lenton_tipping_2008}. 
Generic early-warning indicators have been proposed as model-independent precursors of fold and transcritical bifurcations \citep{dakos_slowing_2008,held_detection_2004,wissel_universal_1984}. 
These indicators include increases in autocorrelation, variance, and return time, a phenomenon known as critical slowing down. 
For Hopf bifurcations specifically, the growing coherence and amplitude of quasi-cycles in the pre-bifurcation regime provide an additional spectral indicator \citep{lade_early_2012,dutta_robustness_2018}.

Practical deployment of early-warning indicators faces several challenges: distinguishing genuine warning signals from statistical artefacts, accounting for non-stationary forcing, and selecting appropriate indicator statistics and sliding-window parameters \citep{dakos_resilience_2015,ditlevsen_tipping_2010,lenton_early_2011}. 
In the predator--prey context, several authors have examined whether variance-based indicators can reliably forecast the onset of critical transitions \citep{krkosek_signals_2014,chen_eigenvalues_2019,oregan_theory_2013}. The stochastic sensitivity ellipse provides a natural geometric framework for such indicators, as it encodes both the magnitude and the directional structure of fluctuations.  The noisy-precursor indicator $\Pi_p$ introduced in this paper (Section~\ref{sec:quasi_cycles}) operationalizes this idea by comparing the principal semi-axis of the confidence ellipse to a characteristic distance from the equilibrium to an extinction-relevant boundary.
Related threshold-based classifications of precursor structure also arise in spatially extended hybrid PDE--ODE systems, where bidirectional chemotactic feedback induces explicit patterning thresholds and resistance-niche formation; see \citep{yu_chemotactic_2026}.

\subsection{Reaction--diffusion patterning and multiscale feedback}

Pattern-forming instabilities in biological systems provide a complementary
perspective on how structural feedback reshapes macroscopic dynamics.
For example, hybrid PDE--agent-based models have demonstrated that
bidirectional endothelial--TAF feedback can generate Turing-type vascular
patterns through reaction--diffusion instability, producing spatially
heterogeneous perfusion and drug-resistant niches \citep{liu_bidirectional_2025}.
Such studies highlight how linear stability thresholds, coupling structure,
and multiscale integration determine emergent spatial organization.
Our focus is temporal stochastic amplification near a Hopf bifurcation. 
However, both frameworks emphasize that structural feedback from spatial Turing mechanisms or stochastic covariance coupling governs the geometry of emergent patterns in complex biological systems.
Related reproducible structured-PDE frameworks in biological systems combine nonlocal operators and threshold effects to quantify regime changes under feedback-like transitions (e.g., dedifferentiation); see \citep{wang_damage-structured_2026}.

\subsection{Covariance structure and noise modeling in stochastic ecology}

Despite the mechanistic clarity of the CTMC/CME-derived diffusion covariance, many stochastic analyses in ecology employ simplified noise structures. A common approach is additive white noise with diagonal (component-independent) covariance \citep{roughgarden_simple_1975,tian_additive_2017,wang_time-delay-induced_2018}, or multiplicative noise proportional to population abundance but still diagonal in the species block \citep{wu_stochastic_2019,pal_nonequilibrium_2025,qi_threshold_2021}. Environmental (extrinsic) stochasticity models likewise typically assume independent or at most correlated-Gaussian perturbations to parameters rather than mechanistically derived covariance \citep{wu_stochastic_2019,pal_nonequilibrium_2025,qi_threshold_2021}.

The importance of correctly specifying the noise covariance structure in population models has been highlighted in several contexts.  Black and McKane \citep{black_stochastic_2012} showed that mechanistic noise derivation is essential for accurate quasi-cycle predictions in predator--prey systems.  Gillespie \citep{gillespie_stochastic_2007} emphasized that the ``chemical Langevin equation'' inherits its noise structure from the CME and that most Langevin equations adopt ad hoc simplifications.
Grima \citep{grima_construction_2011} developed higher-order corrections beyond the LNA that further highlight the sensitivity of moment approximations to covariance modelling.  In the chemical physics literature, the structural relationship between stoichiometry and diffusion covariance is well established \citep{ethier_markov_1986,schnoerr_approximation_2017}, but its consequences for ecological diagnostics near bifurcation thresholds have not been explored in comparable detail.

The general principle that drift equivalence does not imply covariance equivalence is mathematically elementary.
This is because the drift depends on the first moments of channel increments, while the covariance depends on the second moments (see Appendix~\ref{app:covariance_structure_proofs}).
Despite its simplicity, the practical implications of this principle for ecological inference remain underappreciated. 
Two stochastic models sharing the same deterministic backbone can produce quantitatively different stationary covariances, spectral signatures, and stochastic sensitivity geometries simply because their event-level noise coupling differs. 
This paper provides a controlled demonstration of these effects in the canonical R--M setting.

\subsection{Positioning of the present work}

The present paper integrates and extends several of the threads surveyed above.  Its contributions relative to the existing literature can be summarized as follows.  First, we provide a systematic, ``fair-comparison'' framework in which the sole modeling variable is the predation-induced off-diagonal covariance, while all other model ingredients are held fixed. 
Second, we propagate the full non-diagonal diffusion covariance through the complete LNA diagnostic chain.
This chain includes the Lyapunov covariance, the matrix PSD, the SSF ellipse, and the noisy-precursor indicator.
We do not stop at the PSD level, which is a common limitation in the quasi-cycle literature. 
Third, we introduce the dimensionless noisy-precursor indicator $\Pi_p$ as a compact scalar diagnostic that synthesizes fluctuation magnitude, geometry, and proximity to extinction boundaries. 
This provides quantitative evidence that covariance-structure effects on near-Hopf diagnostics are more than theoretical artefacts.
Taken together, these contributions clarify the role of mechanistic noise coupling in shaping the stochastic signatures that precede ecological regime shifts.

\section{Deterministic R--M backbone and enrichment regimes}
\label{sec:deterministic}

This section records the deterministic R--M structure used to organize the stochastic diagnostics developed later. We only retain the ingredients needed for the near-Hopf analysis: the nondimensional model, the coexistence equilibrium, its feasibility conditions, the enrichment-driven Hopf threshold, and the resulting parameter partition.

We consider the nondimensional R--M predator--prey system with Holling type~II predation:
\begin{equation}\label{eq:RM_ODE}
\begin{cases}
\dfrac{dN}{dt} = N\!\left(1-\dfrac{N}{k}\right) - \dfrac{mNP}{1+N}, \\[8pt]
\dfrac{dP}{dt} = P\!\left(-c + \dfrac{mN}{1+N}\right),
\end{cases}
\qquad (N(t),P(t))\in Q \coloneqq \{(N,P):N\ge 0,\ P\ge 0\},
\end{equation}
where \(N\) and \(P\) denote prey and predator densities, respectively. The parameter \(k>0\) is the (scaled) prey carrying capacity, \(c>0\) is the predator mortality parameter, and \(m>0\) is the maximal predation/assimilation parameter under the present nondimensionalization.

The biologically relevant state space is the nonnegative quadrant \(Q\), with interior
\begin{equation}\label{eq:positive_quadrant}
D \coloneqq (0,\infty)^2.
\end{equation}
The stochastic diffusion models studied later will use the same state-space geometry, but with different boundary interpretations (open-domain versus absorbed viewpoints; Section~\ref{sec:formulations}).

For Coexistence equilibrium and feasibility. System~\eqref{eq:RM_ODE} admits the boundary equilibria
\[
K_1=(0,0),\qquad K_2=(k,0),
\]
and the coexistence equilibrium
\begin{equation}\label{eq:K3}
K_3=(N^*,P^*)=
\left(
\frac{c}{m-c},\;
\frac{k(m-c)-c}{k(m-c)^2}
\right).
\end{equation}
The coexistence equilibrium \(K_3\) lies in \(D\) if and only if
\begin{equation}\label{eq:existence_K3}
m>c
\qquad\text{and}\qquad
k(m-c)>c.
\end{equation}
The condition \(m>c\) ensures that predator per-capita growth can become nonnegative at sufficiently high prey density, while \(k(m-c)>c\) is exactly the positivity condition \(P^*>0\).

For enrichment-driven Hopf threshold. Linearization of~\eqref{eq:RM_ODE} at \(K_3\) yields the classical enrichment-driven Hopf threshold
\begin{equation}\label{eq:Hopf_threshold}
k_H=\frac{m+c}{m-c},
\end{equation}
in the parameter regime where \(K_3\in D\). Thus, in the standard R--M setting,
\[
k<k_H \;\Rightarrow\; K_3 \text{ is locally asymptotically stable},
\qquad
k>k_H \;\Rightarrow\; K_3 \text{ is unstable},
\]
with a Hopf bifurcation as \(k\) crosses \(k_H\). This threshold is the deterministic backbone for the near-Hopf stochastic amplification and quasi-cycle diagnostics studied in this paper.

Parameter partition for stochastic analysis. We use the Hopf threshold~\eqref{eq:Hopf_threshold} to define the parameter regions
\begin{equation}\label{eq:Lambda_regions}
\begin{aligned}
\Lambda_2 &\coloneqq \{(m,c,k): k<k_H\}
\quad \text{(deterministically stable coexistence regime)},\\[4pt]
\Lambda_1 &\coloneqq \{(m,c,k): k>k_H\}
\quad \text{(post-Hopf oscillatory regime)}.
\end{aligned}
\end{equation}
The present paper focuses primarily on \(\Lambda_2\), especially parameter values with \(k\uparrow k_H\), where deterministic coexistence remains stable but demographic noise can generate pronounced quasi-cycles and noisy precursors. The post-Hopf regime \(\Lambda_1\) is included mainly for comparisons.

In this paper, the deterministic partition mainly serves as the organizing backbone for stochastic diagnostics.

\section{Stochastic diffusion model and covariance structures}
\label{sec:stochastic_model}

This section states the stochastic diffusion models used throughout the paper in a self-contained form, with emphasis on the covariance structures that drive the near-Hopf diagnostics. Our focus here is not to re-derive the full CTMC/CME construction in detail, but to provide reproducible model definitions for (i) a mechanistically consistent full-covariance demographic diffusion and (ii) a drift-matched diagonal-noise comparator. 

For Mechanistically derived demographic diffusion. Let
\[
x_t=(N_t,P_t)^\top \in \mathbb{R}_+^2
\]
denote prey and predator densities, and let \(\Omega>0\) be the system-size parameter (population scale). We use the density-level It\^o SDE
\begin{equation}\label{eq:stoch_model_SDE}
dx_t=b(x_t)\,dt+\Sigma(x_t)\,dW_t,
\end{equation}
where \(W_t\) is a standard Brownian motion (of dimension determined by the chosen factorization), and the diffusion covariance is
\begin{equation}\label{eq:stoch_model_cov_general}
a(x):=\Sigma(x)\Sigma(x)^\top=\frac{1}{\Omega}\,S\,\operatorname{diag}(f(x))\,S^\top.
\end{equation}
Here \(S\) is a stoichiometric matrix and \(f(x)\) is a vector of density-level event intensities associated with a chosen mechanistic closure. Equation~\eqref{eq:stoch_model_cov_general} is the key modeling identity: the diffusion covariance is inherited from event stoichiometry and channel intensities.

Throughout this paper we use the reduced nondimensional R--M drift
\begin{equation}\label{eq:stoch_model_drift}
b(x)=
\begin{pmatrix}
N\!\left(1-\dfrac{N}{k}\right)-\dfrac{mNP}{1+N}\\[8pt]
P\!\left(-c+\dfrac{mN}{1+N}\right)
\end{pmatrix},
\qquad x=(N,P)^\top.
\end{equation}
This is the same deterministic backbone as in Section~\ref{sec:deterministic}. In the present paper, the stochastic modeling question is how different covariance structures \(a(x)\) (with the same drift~\eqref{eq:stoch_model_drift}) affect near-Hopf second-order diagnostics and agreement with SSA.

For full-covariance closure and diagonal-noise comparator. We compare two diffusion closures that share the same deterministic drift but differ in predation-induced covariance structure:
\begin{enumerate}[label=(\roman*)]
\item Full-covariance model (mechanistic default): predation and predator conversion are coupled at the event level, yielding a nonzero prey--predator noise cross-covariance \(a_{12}(x)\), typically negative.
\item Diagonal-noise comparator (drift-matched baseline): predation-related prey loss and predator gain are represented by split channels whose covariance contribution is diagonal in the prey--predator block, while the deterministic drift and marginal predation intensities are matched.
\end{enumerate}

To make this explicit, write
\[
f_{\mathrm{pred}}(x):=\frac{mNP}{1+N}.
\]
For notational continuity with the mechanistic derivation, we temporarily include a conversion efficiency \(e\in(0,1]\). 

\medskip
\noindent\textbf{(A) Bernoulli-coupled predation--conversion mechanism (exact CTMC-induced covariance contribution).}
If each predation encounter removes one prey and produces a predator offspring with probability \(e\), then the predation contribution to the density diffusion covariance is
\begin{equation}\label{eq:apred_B_paperB}
a_{\mathrm{pred}}^{(B)}(x)
=
\frac{1}{\Omega}\,f_{\mathrm{pred}}(x)
\begin{pmatrix}
1 & -e\\
-e & e
\end{pmatrix}.
\end{equation}
This is the predation-sector covariance inherited from the exact Bernoulli-coupled integer-valued mechanism after diffusion approximation.

\medskip
\noindent\textbf{(B) Effective coupled \texorpdfstring{\((-1,e)^\top\)}{(-1,e)^\top} closure (compact diffusion representation).}
A compact diffusion-level closure replaces the Bernoulli mechanism by a single effective coupled predation channel with increment \((-1,e)^\top\), yielding
\begin{equation}\label{eq:apred_eff_paperB}
a_{\mathrm{pred}}^{(\mathrm{eff})}(x)
=
\frac{1}{\Omega}\,f_{\mathrm{pred}}(x)
\begin{pmatrix}
1 & -e\\
-e & e^2
\end{pmatrix}.
\end{equation}
This closure preserves the same predation-induced cross-covariance sign and magnitude as~\eqref{eq:apred_B_paperB}, and is used as a compact full-covariance representation.

\medskip
\noindent\textbf{(C) Split-channel diagonal-noise comparator (drift-matched baseline).}
If predation-related prey removal and predator reproduction are treated as independent channels while matching deterministic drift and marginal predation intensities, the predation contribution becomes
\begin{equation}\label{eq:apred_split_paperB}
a_{\mathrm{pred}}^{(S)}(x)
=
\frac{1}{\Omega}\,f_{\mathrm{pred}}(x)
\begin{pmatrix}
1 & 0\\
0 & e
\end{pmatrix}.
\end{equation}
This specific formulation yields zero predation-induced cross-covariance between prey and predator.

In all three cases, the full diffusion covariance \(a(x)\) is obtained by adding the predation contribution to the common birth/competition/death contributions implied by~\eqref{eq:stoch_model_cov_general}. Thus the comparator is drift-matched and differs from the full-covariance model specifically in the second-order coupling structure. Figures~\ref{fig:microscopic_mechanism} and~\ref{fig:microscopic_distinction} visualize the modeling distinctions between the full-covariance (A \& B) and the baseline comparator (C) models.

\begin{figure}[htbp]
  \centering
  \includegraphics[width=0.75\textwidth]{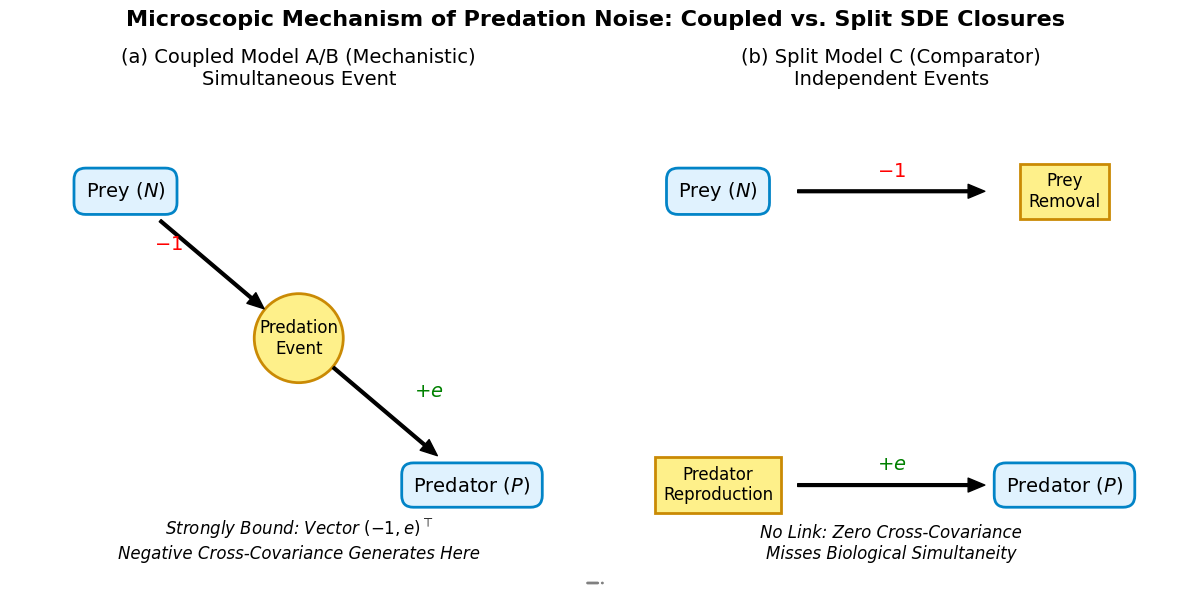}
  \caption{\textbf{Microscopic Mechanism of Predation Noise.} (a) In the coupled closures, a single predation event simultaneously reduces prey and increases predator populations, generating a structurally negative cross-covariance. (b) In the split-channel comparator, these are treated as independent events, yielding zero cross-covariance despite matching the deterministic drift.}
  \label{fig:microscopic_mechanism}
\end{figure}

\begin{figure}[htbp]
  \centering
  \includegraphics[width=0.75\textwidth]{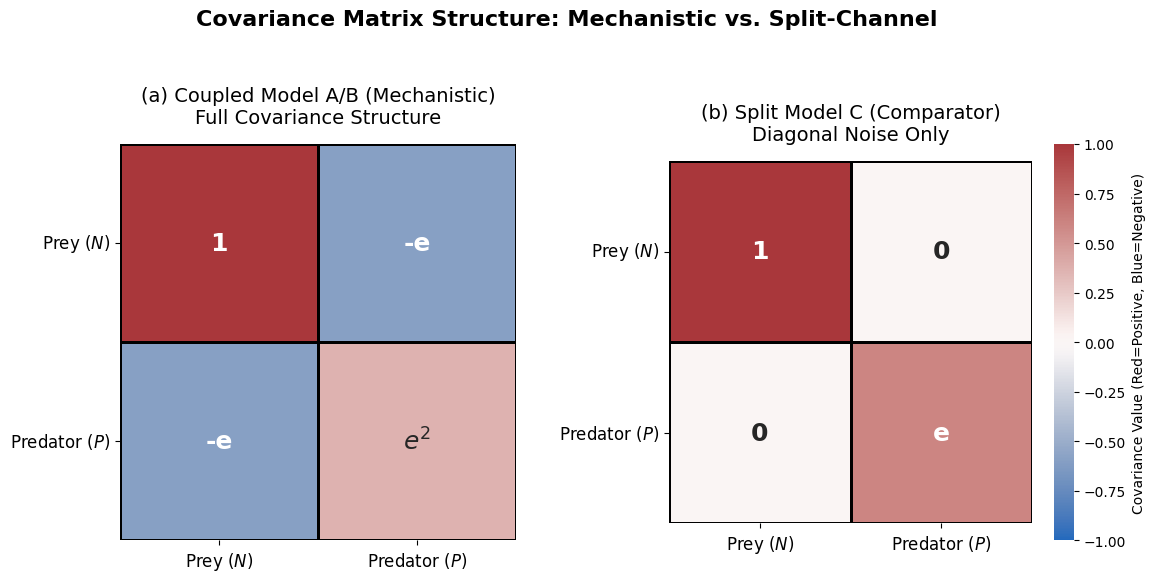}
  \caption{\textbf{Covariance Matrix Structure.} Heatmap representation of the predation contribution to the diffusion covariance. The mechanistic full-covariance model (a) exhibits negative cross-covariance (blue), whereas the split-channel comparator (b) explicitly zeroes out this off-diagonal coupling (white), preserving only the diagonal variances (red).}
  \label{fig:microscopic_distinction}
\end{figure}

Structural predation-induced cross-covariance. The key covariance distinction used throughout the paper is summarized in the following observation.

\begin{proposition}[Structural predation-induced cross-covariance]
\label{prop:structural_crosscov}
Under a coupled predation--conversion closure (Bernoulli-coupled or effective \((-1,e)^\top\) closure), the predation contribution to the diffusion covariance satisfies
\begin{equation}\label{eq:structural_crosscov_negative}
\bigl(a_{\mathrm{pred}}\bigr)_{12}(x)<0,
\qquad x\in(0,\infty)^2.
\end{equation}
In contrast, for the drift-matched split-channel comparator,
\begin{equation}\label{eq:structural_crosscov_zero}
\bigl(a_{\mathrm{pred}}^{(S)}\bigr)_{12}(x)=0,
\qquad x\in(0,\infty)^2.
\end{equation}
\end{proposition}

\begin{proof}
From~\eqref{eq:apred_B_paperB} and~\eqref{eq:apred_eff_paperB},
\[
\bigl(a_{\mathrm{pred}}\bigr)_{12}(x)
=
-\frac{1}{\Omega}\,e\,f_{\mathrm{pred}}(x)
=
-\frac{1}{\Omega}\,e\,\frac{mNP}{1+N}<0
\quad \text{for }N,P>0.
\]
Equation~\eqref{eq:structural_crosscov_zero} follows directly from~\eqref{eq:apred_split_paperB}.
\end{proof}

Proposition~\ref{prop:structural_crosscov} provides the structural basis for the comparisons in later sections.
The LNA covariance, matrix PSD, and stochastic sensitivity ellipse all depend on the local diffusion covariance \(D_*=a(K_3)\).
Consequently, the presence or absence of this cross-covariance directly affects fluctuation geometry, spectral amplification diagnostics, and noisy-precursor indicators near the Hopf threshold.
Recent work has argued systematically that mechanistically derived predator--prey diffusions require non-diagonal demographic noise and a structurally negative prey--predator cross-covariance, and has also formalized the open-domain versus absorbed modeling bifurcation that we discuss in the next section; see \citep{yu_beyond_2026}.

\section{Open-domain and absorbed viewpoints for stochastic interpretation}
\label{sec:formulations}

The diffusion model in Section~\ref{sec:stochastic_model} can be used under two complementary global viewpoints, depending on whether the scientific objective is interior fluctuation analysis or extinction-permitting dynamics. In this paper, this distinction is primarily interpretive and computational: we use it to align the model formulation with the observable of interest, rather than to develop a full boundary well-posedness theory.

We first introduce the open-domain viewpoint for the interior dynamics and near-Hopf diagnostics.
In this viewpoint, the diffusion is interpreted on the interior state space
\[
D=(0,\infty)^2,
\]
and is used to describe survival-conditioned or interior stochastic dynamics near the coexistence equilibrium \(K_3\). This is the natural setting for the local analyses developed later in the paper, including:
\begin{enumerate}[label=(\roman*)]
\item LNA around \(K_3\),
\item stationary covariance (Lyapunov equation) in the Hurwitz regime,
\item matrix PSD and quasi-cycle diagnostics,
\item stochastic sensitivity (SSF) confidence ellipses and the noisy-precursor indicator.
\end{enumerate}
These diagnostics characterize fluctuation amplification and geometry before deterministic oscillatory onset, and they are inherently interior objects centered at the coexistence state.

We discuss the absorbed viewpoint for the extinction-permitting dynamics.
In this viewpoint, the same local diffusion coefficients are used up to boundary contact, but trajectories are treated as extinct (absorbed) once a boundary corresponding to loss of coexistence is reached (i.e.\ \(N=0\) and/or \(P=0\)). This viewpoint is the diffusion-level analogue of the irreversibility present in the finite-population CTMC model, where extinction is absorbing in the absence of immigration or external forcing.

The absorbed viewpoint is therefore the appropriate choice when the target observables involve extinction-permitting behavior, such as absorption probabilities, survival curves, or qualitative comparison of boundary-hitting tendencies with SSA trajectories. 
In the present paper, we use absorbed diffusion viewpoint mainly to compare against SSA in finite-population regimes and to separate extinction-related effects from the interior fluctuation diagnostics obtained under the open-domain viewpoint.
A fully mechanistic full-covariance chemical-Langevin diffusion with absorbing coordinate axes can be derived from a CTMC formulation, with strong well-posedness and moment bounds established up to absorption; see \citep{yu_full-covariance_2026}.

In summary, using both viewpoints allows us to compare stochastic signatures across modeling levels while keeping the interpretation of each diagnostic aligned with the appropriate boundary treatment.

\section{LNA near coexistence: covariance, PSD, stochastic sensitivity, and noisy precursors}
\label{sec:quasi_cycles}

This section develops the main analytical diagnostics of the paper for stochastic fluctuations near the coexistence equilibrium \(K_3\) in the deterministically stable regime \(\Lambda_2\) (Section~\ref{sec:deterministic}). The goal is to quantify how demographic noise is amplified as the enrichment-driven Hopf threshold is approached, and how the covariance structure of that noise affects the resulting quasi-cycle signatures.

A key point of novelty is the explicit propagation of a mechanistically inherited non-diagonal local diffusion covariance \(D_*=a(K_3)\) through the entire LNA diagnostic chain. In much of the quasi-cycle/LNA literature, noise is taken to be diagonal (or its covariance structure is not analyzed explicitly), which can obscure how event-level coupling shapes stochastic observables. Here, the full-covariance model induces a prey--predator cross-covariance at the diffusion level, and this enters directly into the Lyapunov covariance, matrix PSD, and stochastic sensitivity (SSF) ellipse. Consequently, covariance structure affects not only fluctuation magnitude but also fluctuation geometry, including the resulting noisy-precursor indicator.

Throughout this section, we work with the open-domain viewpoint of Section~\ref{sec:formulations} and the stochastic diffusion model of Section~\ref{sec:stochastic_model}, with deterministic drift~\eqref{eq:stoch_model_drift}. The focus is local (near \(K_3\)) and applies to the deterministically stable coexistence regime \(\Lambda_2\), especially near the Hopf threshold \(k_H\).

For LNA / OU reduction around the coexistence equilibrium. Assume the coexistence equilibrium
\[
K_3=(N^*,P^*)
\]
exists in \(D=(0,\infty)^2\) (Section~\ref{sec:deterministic}). Let
\[
J:=J(K_3)
\]
denote the Jacobian of the deterministic drift at \(K_3\), and let
\[
D_*:=a(K_3)=\Sigma(K_3)\Sigma(K_3)^\top
\]
be the local diffusion covariance matrix at \(K_3\) for the chosen noise closure (full-covariance model or diagonal comparator).

Define the fluctuation variable
\[
y_t:=x_t-K_3.
\]
Linearizing the drift around \(K_3\) and freezing the diffusion covariance at \(K_3\) yields the LNA, i.e.\ the OU SDE
\begin{equation}\label{eq:LNA_OU_main}
dy_t = J\,y_t\,dt + B\,dW_t,
\qquad BB^\top=D_*,
\end{equation}
where \(B\) is any matrix factorization of \(D_*\).

In the deterministically stable regime \(\Lambda_2=\{k<k_H\}\), the coexistence equilibrium \(K_3\) is locally asymptotically stable, hence \(J\) is Hurwitz (equivalently, in two dimensions, \(\operatorname{tr}(J)<0\) and \(\det(J)>0\)). Therefore, the OU process~\eqref{eq:LNA_OU_main} admits a unique stationary Gaussian distribution with zero mean and covariance matrix \(W\).

The LNA is a local approximation centered at \(K_3\), and is most informative in \(\Lambda_2\) near \(k_H\), where deterministic attraction persists but stochastic amplification is strong. In particular, the stationary covariance/PSD/SSF diagnostics developed below are defined only in the Hurwitz regime \(\Lambda_2\). In the post-Hopf regime \(\Lambda_1\), \(J(K_3)\) is not Hurwitz and one should not impose a stationary OU interpretation around \(K_3\).

For stationary covariance via the continuous Lyapunov equation.
Let
\[
W \coloneqq \mathbb{E}[y_t y_t^\top]
\]
denote the stationary covariance matrix of the OU process~\eqref{eq:LNA_OU_main} (when \(J\) is Hurwitz). Then \(W\) is the unique symmetric positive definite solution of the continuous Lyapunov equation
\begin{equation}\label{eq:LNA_Lyapunov_main}
J W + W J^\top + D_* = 0.
\end{equation}
This matrix equation is the central covariance object for the stochastic sensitivity analysis.

Since the R--M model is two-dimensional, write
\[
J=\begin{pmatrix} a & b\\ c & d \end{pmatrix},\qquad
D_*=\begin{pmatrix} q_{11} & q_{12}\\ q_{12} & q_{22} \end{pmatrix},\qquad
W=\begin{pmatrix} w_{11} & w_{12}\\ w_{12} & w_{22} \end{pmatrix}.
\]
Then~\eqref{eq:LNA_Lyapunov_main} is equivalent to the \(3\times3\) linear system
\begin{equation}\label{eq:LNA_3x3_main}
\begin{pmatrix}
2a & 2b & 0\\
c & a+d & b\\
0 & 2c & 2d
\end{pmatrix}
\begin{pmatrix}
w_{11}\\ w_{12}\\ w_{22}
\end{pmatrix}
=
-
\begin{pmatrix}
q_{11}\\ q_{12}\\ q_{22}
\end{pmatrix}.
\end{equation}
Hence \(W\) can be obtained by a robust linear solve (or by a standard Lyapunov solver). Appendix~\ref{app:SSF_Lyapunov_2x2} gives the componentwise derivation and explicit \(2\times2\) formulas.

Under the full-covariance stochastic model, \(D_*=a(K_3)\) is generally non-diagonal, reflecting mechanistic prey--predator coupling. Therefore, the stationary covariance \(W\) inherits not only the deterministic linearized dynamics \(J\) but also the noise geometry encoded in \(D_*\). This is precisely where the full-covariance and diagonal-noise models can diverge even when they share the same drift.

The LNA/OU representation also yields a closed-form matrix-valued PSD, which makes quasi-cycle amplification transparent in the frequency domain.

Let \(S(\omega)\in\mathbb{C}^{2\times2}\) denote the stationary matrix PSD of the OU process~\eqref{eq:LNA_OU_main}. Then
\begin{equation}\label{eq:PSD_matrix_main}
S(\omega)=
\bigl(J-i\omega I\bigr)^{-1}D_*\bigl(J^\top+i\omega I\bigr)^{-1},
\qquad \omega\in\mathbb{R}.
\end{equation}
A Fourier-domain derivation is given in Appendix~\ref{app:LNA_PSD}.
The diagonal entries
\[
S_{NN}(\omega),\qquad S_{PP}(\omega)
\]
are the prey and predator fluctuation spectra, while the off-diagonal entries encode cross-spectral coupling and phase information. In \(\Lambda_2\), especially for \(k\uparrow k_H\), one typically observes:
\begin{enumerate}[label=(\roman*),leftmargin=1.8em]
\item emergence and sharpening of a nonzero-frequency spectral peak (quasi-cycles),
\item increased spectral amplification near the resonant frequency as \(\operatorname{tr}(J)\uparrow 0\),
\item covariance-structure dependence of the spectral matrix through \(D_*\), not only through \(J\).
\end{enumerate}
Thus the PSD provides a direct frequency-domain signature of noisy precursors before deterministic destabilization.

At the deterministic Hopf threshold, the real parts of the eigenvalues of \(J\) cross zero. In the stable regime \(\Lambda_2\), the resolvent
\[
(J-i\omega I)^{-1}
\]
becomes increasingly amplifying near the imaginary part of the eigenvalues as \(k\uparrow k_H\), producing the characteristic quasi-cycle spectral peak. The noise covariance \(D_*\) determines how strongly different fluctuation directions are injected into this amplified linear response.

Stochastic sensitivity, confidence ellipses, and a noisy-precursor indicator. The stationary covariance \(W\) obtained from~\eqref{eq:LNA_Lyapunov_main} defines a local stochastic sensitivity geometry around \(K_3\).

For a confidence level \(p\in(0,1)\), define the (deviation-coordinate) confidence ellipse
\begin{equation}\label{eq:SSF_ellipse_main}
\mathcal{E}_p
:=
\left\{
z\in\mathbb{R}^2:\ z^\top W^{-1}z \le \chi^2_2(p)
\right\},
\end{equation}
where \(\chi^2_2(p)\) is the \(p\)-quantile of the \(\chi^2\) distribution with \(2\) degrees of freedom. This is the SSF level set in the LNA/OU approximation.

Let \(\lambda_\pm(W)\) be the eigenvalues of \(W\), ordered so that \(\lambda_+\ge \lambda_->0\). Then the semi-axis lengths of \(\mathcal{E}_p\) are
\begin{equation}\label{eq:SSF_axes_main}
\ell_\pm=\sqrt{\chi^2_2(p)\,\lambda_\pm(W)}.
\end{equation}
The principal-axis orientation is determined by the eigenvectors of \(W\) (equivalently, by the standard \(2\times2\) angle formula; Appendix~\ref{app:SSF_Lyapunov_2x2}).

In the full-covariance model (Model A/B in Section~\ref{sec:stochastic_model}), the non-diagonal \(D_*\) influences \(W\) through the Lyapunov equation~\eqref{eq:LNA_Lyapunov_main}. As a result, the SSF ellipse can exhibit a pronounced tilt relative to the coordinate axes, reflecting prey--predator fluctuation coupling inherited from event-level demographic mechanisms. A diagonal-noise comparator (Model C in Section~\ref{sec:stochastic_model}) may reproduce some marginal variances while still misrepresenting ellipse orientation and anisotropy. Figure~\ref{fig:sensitivity_ellipses} visualizes such a distinction between the full-covariance model and the diagonal comparator.

\begin{figure}[htbp]
  \centering
  \includegraphics[width=0.6\textwidth]{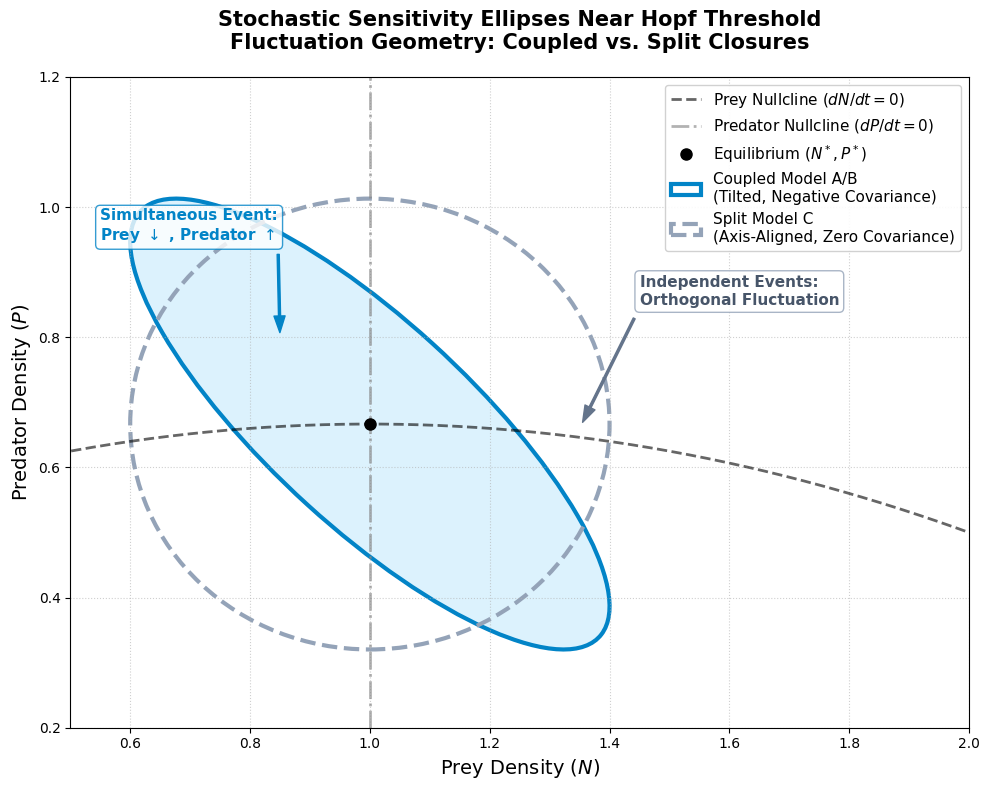}
  \caption{\textbf{Stochastic Sensitivity Ellipses Near the Hopf Threshold.} Fluctuation geometry driven by different diffusion closures. The fully coupled mechanistic model (solid blue) yields a tilted ellipse due to negative cross-covariance, indicating that unexpected prey decreases often coincide with predator increases. The split-channel comparator (dashed gray) incorrectly predicts orthogonal, independent fluctuations.}
  \label{fig:sensitivity_ellipses}
\end{figure}

To quantify how close typical fluctuations are to an extinction-relevant or basin-separating set, let \(\mathcal S\) denote a critical set in deviation coordinates (e.g.\ a local separatrix approximation or extinction-threshold manifold), and define
\[
d_{\mathrm{sep}} \coloneqq \inf_{z\in\mathcal{S}}\|z\|.
\]
We use the dimensionless indicator
\begin{equation}\label{eq:precursor_indicator_main}
\Pi_p \coloneqq \frac{\ell_+}{d_{\mathrm{sep}}}.
\end{equation}
The interpretations include:
\begin{enumerate}[label=(\roman*),leftmargin=1.8em]
\item \(\Pi_p\ll1\): typical fluctuations remain well inside a local safe neighborhood;
\item \(\Pi_p\approx1\): noisy-precursor regime, where confidence-level fluctuations reach the critical set;
\item \(\Pi_p>1\): substantial noise-driven excursions become plausible despite deterministic local stability (\(k<k_H\)).
\end{enumerate}

Figure~\ref{fig:time_series_hopf} shows representative time-series data near the Hopf bifurcation, where Figure~\ref{fig:time_series_hopf}(a) corresponds to the deterministic system with oscillations decaying to the steady state. Figure~\ref{fig:time_series_hopf}(b) demonstrates the system driven by demographic noise, and such noise sustains the amplified quasi-cycles.

\begin{figure}[htbp]
  \centering
  \includegraphics[width=0.65\textwidth]{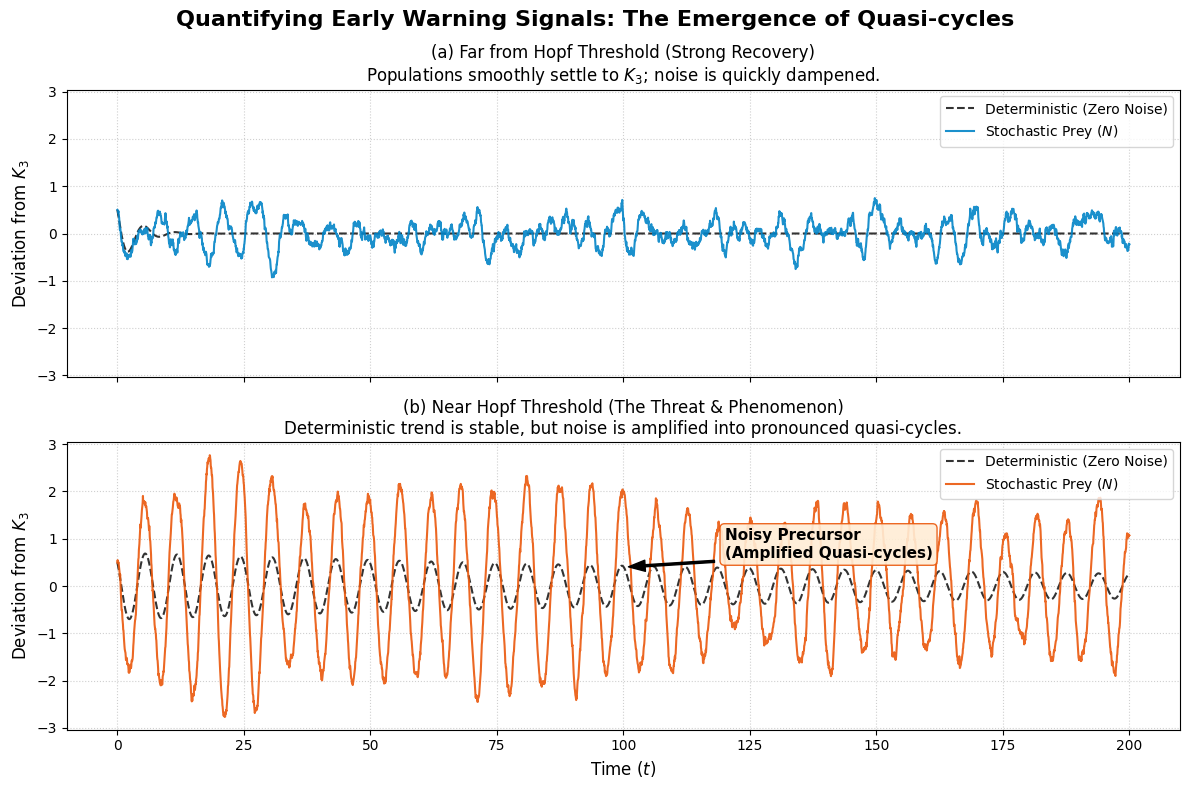}
  \caption{\textbf{Time-Series Near the Hopf Bifurcation.} Simulation of population fluctuations using the linear noise approximation. (Top) The fully coupled mechanistic model (with negative cross-covariance) generates clear, amplified quasi-cycles, serving as a reliable noisy-precursor indicator of the impending Hopf bifurcation. (Bottom) The split-channel comparator (diagonal noise only) fails to effectively excite these resonant cycles, resulting in a muted, random appearance that obscures the early warning signals of system destabilization.}
  \label{fig:time_series_hopf}
\end{figure}

\section{Discussion and conclusion}

This work clarifies a subtle but practically consequential point for stochastic ecological modeling near enrichment-driven Hopf bifurcations: matching the deterministic drift does not guarantee matching the second-order stochastic diagnostics. Starting from a CTMC/CME description of discrete birth--death--predation events, we derived a mechanistically consistent diffusion approximation whose tensor inherits stoichiometric coupling. In the R--M modeling setting, predation is a same-event coupling between prey loss and (probabilistic) predator gain, which induces a structurally negative prey--predator cross-covariance in the local diffusion matrix. We contrasted this full-covariance closure with a drift-matched diagonal-noise comparator that splits predation-related channels and therefore removes the off-diagonal term while preserving the deterministic backbone and marginal noise intensities.

Propagating this single covariance-structure change through the full LNA diagnostic chain revealed systematic differences in near-Hopf behavior. In the stable coexistence regime ($k<k_H$), the OU reduction shows that stationary covariance, matrix-valued PSD, and stochastic sensitivity ellipses depend explicitly on the local diffusion covariance $D_*=a(K_3)$, not just on the Jacobian $J(K_3)$. Consequently, the mechanistic cross-covariance reshapes both the magnitude and the geometry of fluctuations: it alters ellipse orientation/anisotropy, modifies cross-spectral structure, and changes how demographic noise is injected into the resonant linear response that produces quasi-cycles. The dimensionless noisy-precursor indicator $\Pi_p$ further synthesizes these effects by comparing the dominant SSF semi-axis to an extinction-relevant distance scale, making explicit that covariance geometry can shift the inferred proximity to boundary-hitting risk even when deterministic stability persists.

Overall, the main conclusion is that event-level noise geometry is a first-order modeling ingredient for pre-bifurcation inference: simplifying predation to independent channels can misrepresent fluctuation coupling, distort spectral amplification signatures, and bias early-warning or extinction-relevant diagnostics. Methodologically, the paper integrates bifurcation theory with mechanistic diffusion approximations and matrix-valued LNA tools (Lyapunov covariance, PSD, SSF) to provide a controlled, fair-comparison template for assessing noise-structure effects. Substantively, it strengthens the case for stoichiometrically consistent, full-covariance demographic diffusions as the appropriate stochastic counterpart of classical predator--prey theory when finite-population fluctuations and near-threshold warning signals are of interest.

\section*{Use of AI tools declaration}
The authors declare they have not used Artificial Intelligence (AI) tools in the creation of this article.

\section*{Conflict of interest}
The authors declare there are no conflicts of interest.

\begin{appendices}

\section{LNA matrix PSD: Fourier-domain derivation}
\label{app:LNA_PSD}

This appendix derives the matrix-valued PSD formula used in
Section~\ref{sec:quasi_cycles} for the LNA near the coexistence equilibrium.
We make the Fourier convention explicit and record the symmetry properties.

\subsection{OU form of the LNA and stationarity conditions}
\label{app:LNA_PSD_OU_setup}

In the LNA, fluctuations \(y(t)\in\mathbb R^2\) around the coexistence equilibrium satisfy the OU SDE
\begin{equation}\label{eq:appB_OU}
dy(t)=J\,y(t)\,dt+B\,dW(t),
\end{equation}
where
\begin{enumerate}[label=(\roman*)]
\item \(J\in\mathbb R^{2\times2}\) is the Jacobian matrix evaluated at the coexistence equilibrium,
\item \(B\in\mathbb R^{2\times r}\) is any matrix factorization of the local diffusion covariance \(D_*\), i.e.
\[
D_* = BB^\top,
\]
\item \(W(t)\in\mathbb R^r\) is a standard \(r\)-dimensional Brownian motion.
\end{enumerate}

The PSD formula in the main text is a stationary OU result and therefore applies only when \(J\) is Hurwitz (all eigenvalues have strictly negative real parts). In the present paper, this corresponds to the deterministically stable coexistence regime \(\Lambda_2\) (Section~\ref{sec:deterministic}). If \(J\) is not Hurwitz (e.g.\ in the post-Hopf regime \(\Lambda_1\)), the stationary covariance and stationary PSD around the coexistence equilibrium are not defined.

\subsection{Fourier transform convention and matrix PSD definition}
\label{app:LNA_PSD_convention}

We use the Fourier transform convention
\begin{equation}\label{eq:appB_FT_def}
\widehat{u}(\omega)
=
\int_{-\infty}^{\infty} e^{-i\omega t}u(t)\,dt,
\qquad
u(t)=\frac{1}{2\pi}\int_{-\infty}^{\infty} e^{i\omega t}\widehat{u}(\omega)\,d\omega.
\end{equation}

For a weakly stationary vector process \(y(t)\), define the covariance function
\begin{equation}\label{eq:appB_covfun}
R(\tau):=\mathbb E\!\big[y(t+\tau)y(t)^\top\big],
\end{equation}
which is independent of \(t\). The matrix-valued PSD is the Fourier transform of \(R(\tau)\):
\begin{equation}\label{eq:appB_PSD_def}
S(\omega)=\int_{-\infty}^{\infty} e^{-i\omega\tau}R(\tau)\,d\tau,
\qquad
R(\tau)=\frac{1}{2\pi}\int_{-\infty}^{\infty} e^{i\omega\tau}S(\omega)\,d\omega.
\end{equation}

Under this convention, no additional \(2\pi\) factor appears in the forward transform for \(S(\omega)\); the \(2\pi\) factor appears only in the inverse transform.

\subsection{Transfer-function derivation of the matrix PSD}
\label{app:LNA_PSD_derivation}

We derive the PSD formula via the standard transfer-function representation of the stationary OU process.
Formally differentiating~\eqref{eq:appB_OU} in the sense of generalized stochastic processes gives
\begin{equation}\label{eq:appB_OU_formal}
\dot y(t)=J\,y(t)+B\,\xi(t),
\end{equation}
where \(\xi(t):=\dot W(t)\) is \(r\)-dimensional Gaussian white noise with covariance
\begin{equation}\label{eq:appB_white_noise_cov}
\mathbb E\!\big[\xi(t)\xi(s)^\top\big]=\delta(t-s)I_r.
\end{equation}

Taking Fourier transforms of~\eqref{eq:appB_OU_formal} yields
\[
(i\omega I-J)\widehat y(\omega)=B\,\widehat\xi(\omega),
\]
hence
\begin{equation}\label{eq:appB_transfer_function}
\widehat y(\omega)=H(\omega)\widehat\xi(\omega),
\qquad
H(\omega):=(i\omega I-J)^{-1}B.
\end{equation}

Since the white-noise input has (matrix) spectral density \(I_r\), the output PSD is
\begin{equation}\label{eq:appB_PSD_transfer}
S(\omega)=H(\omega)\,I_r\,H(\omega)^\ast
=
\big[(i\omega I-J)^{-1}B\big]\big[(i\omega I-J)^{-1}B\big]^\ast,
\end{equation}
where \(^\ast\) denotes conjugate transpose.

Using \(D_*=BB^\top\) and the identities
\[
(i\omega I-J)^{-1}=-(J-i\omega I)^{-1},
\qquad
\big[(i\omega I-J)^{-1}\big]^\ast
=
(-i\omega I-J^\top)^{-1}
=
-(J^\top+i\omega I)^{-1},
\]
the minus signs cancel, yielding the matrix PSD formula used in the main text:
\begin{equation}\label{eq:appB_PSD_final}
S(\omega)=
\bigl(J-i\omega I\bigr)^{-1}D_*
\bigl(J^\top+i\omega I\bigr)^{-1},
\qquad \omega\in\mathbb R.
\end{equation}

Because \(J\) and \(D_*\) are real matrices and \(D_*=D_*^\top\), the matrix PSD satisfies the Hermitian symmetry
\begin{equation}\label{eq:appB_Hermitian_symmetry}
S(-\omega)=S(\omega)^\ast.
\end{equation}
Equivalently, the real parts of diagonal entries are even functions of \(\omega\), and the imaginary parts of off-diagonal entries are odd (with the corresponding conjugacy relations).

\section{Stationary covariance from the \texorpdfstring{\(2\times 2\)}{2x2} Lyapunov equation and SSF ellipse formulas}
\label{app:SSF_Lyapunov_2x2}

This appendix provides an explicit scheme for the
stationary covariance in the \(2\times2\) LNA/OU system, together with the stochastic
sensitivity function (SSF) ellipse formulas and the noisy-precursor indicator
\(\Pi_p\) used in the main text.
The notation is fully aligned with Section~\ref{sec:quasi_cycles}: \(J\) denotes the
Jacobian at the coexistence equilibrium, \(D_*\) the local diffusion covariance, and
\(W\) the stationary covariance of the LNA fluctuations.

\subsection{LNA/OU setup and Lyapunov equation}
\label{app:SSF_setup}

In the deterministically stable coexistence regime (\(J\) Hurwitz), the LNA fluctuations
\(y_t=x_t-K_3\) satisfy the OU SDE
\begin{equation}\label{eq:appC_OU}
dy_t = J\,y_t\,dt + B\,dW_t,
\qquad BB^\top = D_*,
\end{equation}
where \(J,D_*\in\mathbb R^{2\times2}\), \(D_*=D_*^\top\), and \(W_t\) is a standard Brownian motion of suitable dimension.

The stationary covariance matrix
\[
W \coloneqq \mathbb{E}[y_t y_t^\top]
=
\begin{pmatrix}
w_{11} & w_{12}\\
w_{12} & w_{22}
\end{pmatrix}
\]
is the unique symmetric solution of the continuous Lyapunov equation
\begin{equation}\label{eq:appC_Lyapunov}
J W + W J^\top + D_* = 0.
\end{equation}

Write
\begin{equation}\label{eq:appC_entries}
J=\begin{pmatrix} a & b \\ c & d \end{pmatrix},
\qquad
D_*=\begin{pmatrix} q_{11} & q_{12} \\ q_{12} & q_{22} \end{pmatrix}.
\end{equation}

Expanding~\eqref{eq:appC_Lyapunov} entrywise and using symmetry of \(W\) gives
\begin{subequations}\label{eq:appC_componentwise}
\begin{align}
2a\,w_{11} + 2b\,w_{12} &= -q_{11}, \label{eq:appC_11}\\
c\,w_{11} + (a+d)\,w_{12} + b\,w_{22} &= -q_{12}, \label{eq:appC_12}\\
2c\,w_{12} + 2d\,w_{22} &= -q_{22}. \label{eq:appC_22}
\end{align}
\end{subequations}

Equivalently,
\begin{equation}\label{eq:appC_linear_system}
\underbrace{\begin{pmatrix}
2a & 2b & 0\\
c & a+d & b\\
0 & 2c & 2d
\end{pmatrix}}_{=: \mathcal M}
\begin{pmatrix}
w_{11}\\ w_{12}\\ w_{22}
\end{pmatrix}
=
-
\begin{pmatrix}
q_{11}\\ q_{12}\\ q_{22}
\end{pmatrix}.
\end{equation}
Thus
\begin{equation}\label{eq:appC_w_solution}
\begin{pmatrix}
w_{11}\\ w_{12}\\ w_{22}
\end{pmatrix}
=
-\mathcal M^{-1}
\begin{pmatrix}
q_{11}\\ q_{12}\\ q_{22}
\end{pmatrix},
\end{equation}
whenever \(\mathcal M\) is invertible.

A direct computation yields
\begin{equation}\label{eq:appC_detM}
\det(\mathcal M)=4(a+d)(ad-bc)=4\,\operatorname{tr}(J)\,\det(J).
\end{equation}
Hence, if \(J\) is Hurwitz (in the \(2\times2\) case equivalently \(\operatorname{tr}(J)<0\) and \(\det(J)>0\)),
then \(\det(\mathcal M)\neq 0\), and the Lyapunov equation~\eqref{eq:appC_Lyapunov}
has a unique symmetric solution \(W\).
This is the regime in which the stationary LNA covariance, PSD, and SSF diagnostics are defined in the main text (Section~\ref{sec:quasi_cycles}).

\subsection{SSF confidence ellipse and principal-axis geometry}
\label{app:SSF_ellipse}

For a confidence level \(p\in(0,1)\), the SSF confidence ellipse in deviation coordinates is
\begin{equation}\label{eq:appC_conf_ellipse}
\mathcal E_p
=
\left\{
z\in\mathbb R^2:\;
z^\top W^{-1} z \le \chi^2_2(p)
\right\},
\end{equation}
where \(\chi^2_2(p)\) is the \(p\)-quantile of the chi-square distribution with \(2\) degrees of freedom.

Let \(\lambda_+(W)\ge \lambda_-(W)>0\) be the eigenvalues of \(W\). For a \(2\times2\) symmetric matrix,
\begin{equation}\label{eq:appC_eigs}
\lambda_{\pm}(W)
=
\frac{w_{11}+w_{22}}{2}
\pm
\frac{1}{2}\sqrt{(w_{11}-w_{22})^2+4w_{12}^2}.
\end{equation}
The corresponding semi-axis lengths are
\begin{equation}\label{eq:appC_axes}
\ell_\pm=\sqrt{\chi^2_2(p)\,\lambda_\pm(W)}.
\end{equation}

A robust expression for the principal-axis angle \(\theta\) (measured from the \(N\)-axis to the major axis of the ellipse) is
\begin{equation}\label{eq:appC_theta_atan2}
\theta=\frac{1}{2}\operatorname{atan2}\!\bigl(2w_{12},\,w_{11}-w_{22}\bigr),
\end{equation}
where \(\operatorname{atan2}(y,x)\in(-\pi,\pi]\) is the two-argument arctangent.
This is preferable to using \(\tan(2\theta)=2w_{12}/(w_{11}-w_{22})\) directly, because \(\operatorname{atan2}\) handles the quadrant correctly and is more robust when \(w_{11}\approx w_{22}\).

Alternatively, \(\theta\) may be obtained from the normalized eigenvector associated with \(\lambda_+(W)\). The two definitions agree up to the sign/orientation convention \(\theta\sim \theta+\pi\).

\subsection{Noisy-precursor indicator \texorpdfstring{\(\Pi_p\)}{Pi\_p}}
\label{app:SSF_precursor}

The noisy-precursor indicator used in the main text is
\begin{equation}\label{eq:appC_precursor_ratio}
\Pi_p=\frac{\ell_+}{d_{\mathrm{sep}}},
\end{equation}
where \(\ell_+\) is the major semi-axis length from~\eqref{eq:appC_axes} and \(d_{\mathrm{sep}}\) is a problem-dependent distance from the equilibrium to an extinction-relevant or basin-separating boundary (in deviation coordinates).

\subsection{Implementation recipe (Lyapunov to SSF to noisy-precursor indicator)}
\label{app:SSF_algorithm}

For reproducibility, we summarize the steps used in the paper in Algorithm~\ref{alg}.

\begin{algorithm}[h]\label{alg}
\caption{Computation of \(W\), SSF ellipse, and noisy-precursor indicator in the \(2\times2\) LNA}
\label{alg:appC_ssf}
\begin{algorithmic}[1]
\Require Coexistence equilibrium \(K_3\), Jacobian \(J=J(K_3)\), local covariance \(D_*=a(K_3)\), confidence level \(p\)
\Ensure Stationary covariance \(W\), ellipse semi-axes \(\ell_\pm\), angle \(\theta\), indicator \(\Pi_p\) (if \(d_{\mathrm{sep}}\) specified)
\State Check Hurwitz condition for \(J\) (equivalently in \(2\times2\): \(\operatorname{tr}(J)<0\) and \(\det(J)>0\))
\If{Hurwitz fails}
    \State Return ``stationary LNA diagnostics not defined''
\EndIf
\State Form \(\mathcal M\) in~\eqref{eq:appC_linear_system} and solve for \((w_{11},w_{12},w_{22})^\top\)
\State Construct \(W=\begin{psmallmatrix} w_{11}&w_{12}\\ w_{12}&w_{22}\end{psmallmatrix}\) and (optionally) symmetrize numerically
\State Compute eigenvalues \(\lambda_\pm(W)\) from~\eqref{eq:appC_eigs}
\State Compute semi-axis lengths \(\ell_\pm\) from~\eqref{eq:appC_axes}
\State Compute major-axis angle \(\theta=\tfrac12\operatorname{atan2}(2w_{12},\,w_{11}-w_{22})\)
\If{\(d_{\mathrm{sep}}\) is specified}
    \State Compute \(\Pi_p=\ell_+/d_{\mathrm{sep}}\)
\EndIf
\end{algorithmic}
\end{algorithm}

In implementation, it is useful to verify that \(W\) is positive definite (within numerical tolerance) and that \(\lambda_\pm(W)>0\) in the Hurwitz regime. Small symmetry violations due to floating-point roundoff can be removed by replacing \(W\) with \((W+W^\top)/2\) before spectral decomposition.

\section{Implementation-level summary of covariance closures: full-covariance model versus diagonal comparator}
\label{app:covariance_closures}

This appendix summarizes the diffusion coefficients, with emphasis on the precise definition of the full-covariance model and the diagonal-noise comparator.
The goal is modeling transparency: both models use the same deterministic drift and differ only in the diffusion covariance structure associated with predation--conversion events.

\subsection{Common drift and notation}
\label{app:cov_common_drift}

Let \(x=(N,P)^\top\in (0,\infty)^2\) denote prey and predator densities, and let \(\Omega>0\) be the system-size parameter.
Throughout the paper, both stochastic models (full covariance and diagonal comparator) use the same reduced R--M drift:
\begin{equation}\label{eq:appD_common_drift}
b(x)=
\begin{pmatrix}
N\!\left(1-\dfrac{N}{k}\right)-\dfrac{mNP}{1+N}\\[8pt]
P\!\left(-c+\dfrac{mN}{1+N}\right)
\end{pmatrix}.
\end{equation}
Define the predation intensity
\begin{equation}\label{eq:appD_fpred}
f_{\mathrm{pred}}(x):=\frac{mNP}{1+N}.
\end{equation}

For compact comparison, we decompose the diffusion covariance into a shared non-predation part and a predation-induced part:
\begin{equation}\label{eq:appD_cov_decomp}
a(x)=a_0(x)+a_{\mathrm{pred}}(x).
\end{equation}
When we fix (\(e=1\)), the shared non-predation contribution is
\begin{equation}\label{eq:appD_a0_e1}
a_0(x)
=
\frac{1}{\Omega}
\begin{pmatrix}
N+\dfrac{N^2}{k} & 0\\[8pt]
0 & cP
\end{pmatrix},
\end{equation}
corresponding to prey birth, prey competition death, and predator death channels.

\subsection{General-\texorpdfstring{\(e\)}{e} formulas}
\label{app:cov_general_e}

We first record the covariance formulas with explicit conversion efficiency \(e\in(0,1]\). These expressions clarify the relationship between the coupled and split-channel closures.

\medskip
\noindent\textbf{(i) Bernoulli-coupled exact-CTMC-induced diffusion covariance contribution.}

For the Bernoulli-coupled predation--conversion mechanism (successful conversion with probability \(e\)),
the predation contribution to the density diffusion covariance is
\begin{equation}\label{eq:appD_apred_B}
a_{\mathrm{pred}}^{(B)}(x)
=
\frac{1}{\Omega}\,f_{\mathrm{pred}}(x)
\begin{pmatrix}
1 & -e\\
-e & e
\end{pmatrix}.
\end{equation}

\medskip
\noindent\textbf{(ii) Effective coupled \texorpdfstring{\((-1,e)^\top\)}{(-1,e)^\top} closure (compact diffusion closure).}

For the effective coupled diffusion-level closure with stoichiometric increment \(\nu_{\mathrm{eff}}=(-1,e)^\top\),
the predation covariance contribution is
\begin{equation}\label{eq:appD_apred_eff}
a_{\mathrm{pred}}^{(\mathrm{eff})}(x)
=
\frac{1}{\Omega}\,f_{\mathrm{pred}}(x)
\begin{pmatrix}
1 & -e\\
-e & e^2
\end{pmatrix}.
\end{equation}

\medskip
\noindent\textbf{(iii) Drift-matched split-channel (diagonal-noise comparator).}

For the split-channel comparator (independent prey-removal and predator-birth channels, drift-matched to the same deterministic model),
the predation covariance contribution is
\begin{equation}\label{eq:appD_apred_split}
a_{\mathrm{pred}}^{(S)}(x)
=
\frac{1}{\Omega}\,f_{\mathrm{pred}}(x)
\begin{pmatrix}
1 & 0\\
0 & e
\end{pmatrix}.
\end{equation}

The coupled closures \eqref{eq:appD_apred_B}--\eqref{eq:appD_apred_eff} have a strictly negative prey--predator cross-covariance for \(N,P>0\),
\[
\bigl(a_{\mathrm{pred}}\bigr)_{12}(x)<0,
\]
whereas the split-channel comparator \eqref{eq:appD_apred_split} has
\[
\bigl(a_{\mathrm{pred}}^{(S)}\bigr)_{12}(x)=0.
\]
This is the key covariance-structure difference carried into the Lyapunov/PSD/SSF diagnostics.

\subsection{Closures used in this paper and ``fair comparison'' definition}
\label{app:cov_fair_comparison}

In the main text, we compare:

\begin{enumerate}[label=(\roman*)]
\item a full-covariance diffusion model (mechanistically coupled predation--conversion closure), and
\item a diagonal-noise comparator (drift-matched split-channel closure).
\end{enumerate}

The comparison is fair in the following sense:
\begin{enumerate}[label=(\roman*)]
\item the deterministic drift \(b(x)\) is identical in both models \eqref{eq:appD_common_drift};
\item the same parameter values \((m,c,k,\Omega)\), initial conditions, time discretization, burn-in/sampling windows, and Monte Carlo replication counts are used;
\item the only modeling difference is the covariance structure of the diffusion term (specifically, the predation-induced off-diagonal covariance and associated covariance geometry).
\end{enumerate}

\subsection{\texorpdfstring{\(e=1\)}{e=1} specialization in the manuscript:}
\label{app:cov_e1_specialization}

We fix \(e=1\).
In this case, the Bernoulli-coupled and effective coupled closures coincide at the predation-covariance level:
\begin{equation}\label{eq:appD_apred_full_e1}
a_{\mathrm{pred}}^{\mathrm{full}}(x)
=
\frac{1}{\Omega}\,f_{\mathrm{pred}}(x)
\begin{pmatrix}
1 & -1\\
-1 & 1
\end{pmatrix},
\qquad
f_{\mathrm{pred}}(x)=\frac{mNP}{1+N}.
\end{equation}
Hence the full-covariance diffusion matrix is
\begin{equation}\label{eq:appD_a_full_e1}
a^{\mathrm{full}}(x)
=
a_0(x)+a_{\mathrm{pred}}^{\mathrm{full}}(x)
=
\frac{1}{\Omega}
\begin{pmatrix}
N+\dfrac{N^2}{k}+\dfrac{mNP}{1+N} & -\dfrac{mNP}{1+N}\\[10pt]
-\dfrac{mNP}{1+N} & cP+\dfrac{mNP}{1+N}
\end{pmatrix}.
\end{equation}

The drift-matched diagonal-noise comparator is
\begin{equation}\label{eq:appD_a_diag_e1}
a^{\mathrm{diag}}(x)
=
a_0(x)+a_{\mathrm{pred}}^{(S)}(x)\big|_{e=1}
=
\frac{1}{\Omega}
\begin{pmatrix}
N+\dfrac{N^2}{k}+\dfrac{mNP}{1+N} & 0\\[10pt]
0 & cP+\dfrac{mNP}{1+N}
\end{pmatrix}.
\end{equation}

Comparing \eqref{eq:appD_a_full_e1} and \eqref{eq:appD_a_diag_e1}, the diagonal entries are identical, and the sole difference is the predation-induced cross-covariance:
\begin{equation}\label{eq:appD_only_difference_e1}
a^{\mathrm{full}}_{12}(x)=a^{\mathrm{full}}_{21}(x)
=
-\frac{1}{\Omega}\frac{mNP}{1+N},
\qquad
a^{\mathrm{diag}}_{12}(x)=a^{\mathrm{diag}}_{21}(x)=0.
\end{equation}
Thus, differences in stationary covariance, PSD, SSF ellipse orientation, and noisy-precursor indicators can be attributed directly to the presence or absence of mechanistically inherited prey--predator cross-covariance, rather than to drift mismatch or unequal marginal noise intensities.

\section{Explicit Jacobian at the coexistence equilibrium \texorpdfstring{$K_3$}{K3}}
\label{app:Jacobian_K3}

Because the LNA, Lyapunov covariance, and PSD diagnostics all depend on
\[
J = J(K_3),
\]
we record here the explicit Jacobian entries for the nondimensional R--M system used throughout the paper and verify the Hopf-threshold sign change of \(\operatorname{tr}(J(K_3))\).

\subsection{Deterministic drift and general Jacobian}
\label{app:Jacobian_general}

The deterministic R--M system is
\begin{equation}\label{eq:appJ_RM}
\begin{cases}
\dfrac{dN}{dt} = f(N,P) = N\!\left(1-\dfrac{N}{k}\right)-\dfrac{mNP}{1+N}, \\[8pt]
\dfrac{dP}{dt} = g(N,P) = P\!\left(-c+\dfrac{mN}{1+N}\right).
\end{cases}
\end{equation}
Its Jacobian matrix is
\begin{equation}\label{eq:appJ_general_J}
J(N,P)=
\begin{pmatrix}
f_N(N,P) & f_P(N,P)\\[4pt]
g_N(N,P) & g_P(N,P)
\end{pmatrix}
=
\begin{pmatrix}
1-\dfrac{2N}{k}-\dfrac{mP}{(1+N)^2} & -\dfrac{mN}{1+N}\\[10pt]
\dfrac{mP}{(1+N)^2} & -c+\dfrac{mN}{1+N}
\end{pmatrix}.
\end{equation}

\subsection{Coexistence equilibrium and useful identities}
\label{app:Jacobian_K3_identities}

When the coexistence equilibrium exists, it is
\begin{equation}\label{eq:appJ_K3}
K_3=(N^*,P^*)=
\left(
\frac{c}{m-c},
\frac{k(m-c)-c}{k(m-c)^2}
\right),
\end{equation}
under the feasibility conditions \(m>c\) and \(k(m-c)>c\).

At \(K_3\), the following identities are repeatedly useful:
\begin{equation}\label{eq:appJ_identities}
1+N^* = 1+\frac{c}{m-c} = \frac{m}{m-c},
\qquad
\frac{N^*}{1+N^*}=\frac{c}{m},
\qquad
\frac{1}{(1+N^*)^2}=\frac{(m-c)^2}{m^2}.
\end{equation}
Moreover, the predator isocline condition implies
\begin{equation}\label{eq:appJ_gP_zero}
g_P(K_3)= -c+\frac{mN^*}{1+N^*}=0.
\end{equation}

\subsection{Explicit entries of \texorpdfstring{$J(K_3)$}{J(K3)}}
\label{app:Jacobian_K3_entries}

Write
\begin{equation}\label{eq:appJ_abcd_def}
J(K_3)=
\begin{pmatrix}
a & b\\
c_J & d
\end{pmatrix},
\end{equation}
where we denote the \((2,1)\)-entry by \(c_J\) (to avoid confusion with the model parameter \(c\)).

Substituting \((N^*,P^*)\) into~\eqref{eq:appJ_general_J} gives:
\begin{align}
b &= -\frac{mN^*}{1+N^*} = -c, \label{eq:appJ_b}\\[4pt]
d &= -c+\frac{mN^*}{1+N^*}=0, \label{eq:appJ_d}
\end{align}
and
\begin{align}
c_J
&= \frac{mP^*}{(1+N^*)^2}
= \frac{m}{(1+N^*)^2}\cdot \frac{k(m-c)-c}{k(m-c)^2}
= \frac{k(m-c)-c}{km}. \label{eq:appJ_cJ}
\end{align}

For the \((1,1)\)-entry, one may simplify directly or use \(mP^*/(1+N^*)^2=c_J\):
\begin{align}
a
&=1-\frac{2N^*}{k}-\frac{mP^*}{(1+N^*)^2}\notag\\
&=1-\frac{2c}{k(m-c)}-\frac{k(m-c)-c}{km}\notag\\
&=\frac{c\,[\,k(m-c)-(m+c)\,]}{km(m-c)}. \label{eq:appJ_a}
\end{align}

Hence
\begin{equation}\label{eq:appJ_JK3_explicit}
J(K_3)=
\begin{pmatrix}
\dfrac{c\,[\,k(m-c)-(m+c)\,]}{km(m-c)} & -c\\[12pt]
\dfrac{k(m-c)-c}{km} & 0
\end{pmatrix}.
\end{equation}

\begin{remark}[Equivalent notation used in the \(2\times2\) Lyapunov appendix.]\label{rmk:equivalent_notation}
If one writes \(J=\begin{psmallmatrix} a & b\\ c & d\end{psmallmatrix}\) in Appendix~\ref{app:SSF_Lyapunov_2x2},
then the identification is
\begin{equation}\label{eq:appJ_abcd_identification}
a=\frac{c\,[\,k(m-c)-(m+c)\,]}{km(m-c)},
\qquad
b=-c,
\qquad
c=\frac{k(m-c)-c}{km},
\qquad
d=0.
\end{equation}
\end{remark}

\subsection{Trace, determinant, and Hopf-threshold sign change}
\label{app:Jacobian_trace_det}

From~\eqref{eq:appJ_JK3_explicit},
\begin{equation}\label{eq:appJ_trace}
\operatorname{tr}\bigl(J(K_3)\bigr)=a
=
\frac{c\,[\,k(m-c)-(m+c)\,]}{km(m-c)}.
\end{equation}
Therefore, defining
\begin{equation}\label{eq:appJ_kH}
k_H=\frac{m+c}{m-c},
\end{equation}
we obtain the factorized form
\begin{equation}\label{eq:appJ_trace_factorized}
\operatorname{tr}\bigl(J(K_3)\bigr)
=
\frac{c}{km(m-c)}\,(m-c)\,(k-k_H)
=
\frac{c}{km}\,(k-k_H).
\end{equation}
Since \(c>0\), \(k>0\), and \(m>0\), this makes the sign change explicit:
\begin{equation}\label{eq:appJ_trace_sign}
\operatorname{tr}\bigl(J(K_3)\bigr)<0 \iff k<k_H,
\qquad
\operatorname{tr}\bigl(J(K_3)\bigr)=0 \iff k=k_H,
\qquad
\operatorname{tr}\bigl(J(K_3)\bigr)>0 \iff k>k_H.
\end{equation}

The determinant is
\begin{equation}\label{eq:appJ_det}
\det\bigl(J(K_3)\bigr)=ad-bc_J = -(-c)c_J = c\,c_J
= \frac{c\,[\,k(m-c)-c\,]}{km},
\end{equation}
which is strictly positive under the coexistence feasibility condition \(k(m-c)>c\).
Thus, in the coexistence regime, local stability is governed by the sign of the trace, and the enrichment-driven Hopf threshold occurs at \(k=k_H\), as used throughout the main text.

\section{Technical proofs for the covariance-structure results in Section~\ref{sec:stochastic_model}}
\label{app:covariance_structure_proofs}

This appendix provides detailed calculations and proofs for the structural covariance results stated in Section~\ref{sec:stochastic_model}, with particular emphasis on the mechanistic origin of the negative prey--predator cross-covariance under coupled predation--conversion closures. The purpose is to make the key claims in the main text fully verifiable and citable.

Throughout, we write
\[
x=(N,P)\in D=(0,\infty)^2,
\qquad
f_{\mathrm{pred}}(x):=\frac{mNP}{1+N},
\]
and recall the density-level covariance identity
\begin{equation}
\label{eq:appBcov_unified_cov}
a(x)=\frac{1}{\Omega}\sum_{k=1}^K f_k(x)\,\nu_k\nu_k^\top
=\frac{1}{\Omega}\,S\,\mathrm{diag}(f(x))\,S^\top.
\end{equation}
In this appendix, \(a_{\mathrm{pred}}(\cdot)\) denotes the contribution to \(a(\cdot)\) coming only from predation-related channels.

\subsection{Predation covariance under the Bernoulli-coupled closure}
\label{app:bern_covariance_derivation}

Recall the Bernoulli-coupled exact CTMC predation channels used in Section~\ref{sec:stochastic_model}:
\[
\nu_4^{(B)}=\binom{-1}{1},
\qquad
\nu_5^{(B)}=\binom{-1}{0},
\]
with density-level intensities
\[
f_4^{(B)}(x)=e\,f_{\mathrm{pred}}(x),
\qquad
f_5^{(B)}(x)=(1-e)\,f_{\mathrm{pred}}(x),
\qquad e\in(0,1].
\]
Therefore, the predation contribution to the diffusion covariance is
\begin{equation}
\label{eq:appBcov_apred_B_start}
a_{\mathrm{pred}}^{(B)}(x)
=
\frac{1}{\Omega}
\left[
f_4^{(B)}(x)\,\nu_4^{(B)}(\nu_4^{(B)})^\top
+
f_5^{(B)}(x)\,\nu_5^{(B)}(\nu_5^{(B)})^\top
\right].
\end{equation}
We compute the rank-one matrices explicitly:
\begin{align}
\nu_4^{(B)}(\nu_4^{(B)})^\top
&=
\binom{-1}{1}\begin{pmatrix}-1 & 1\end{pmatrix}
=
\begin{pmatrix}
1 & -1\\
-1 & 1
\end{pmatrix}, \label{eq:appBcov_nu4_outer}\\[6pt]
\nu_5^{(B)}(\nu_5^{(B)})^\top
&=
\binom{-1}{0}\begin{pmatrix}-1 & 0\end{pmatrix}
=
\begin{pmatrix}
1 & 0\\
0 & 0
\end{pmatrix}. \label{eq:appBcov_nu5_outer}
\end{align}
Substituting \eqref{eq:appBcov_nu4_outer}--\eqref{eq:appBcov_nu5_outer} into \eqref{eq:appBcov_apred_B_start} gives
\begin{align}
a_{\mathrm{pred}}^{(B)}(x)
&=
\frac{1}{\Omega}\left[
e f_{\mathrm{pred}}(x)
\begin{pmatrix}
1 & -1\\
-1 & 1
\end{pmatrix}
+
(1-e) f_{\mathrm{pred}}(x)
\begin{pmatrix}
1 & 0\\
0 & 0
\end{pmatrix}
\right] \notag\\[4pt]
&=
\frac{f_{\mathrm{pred}}(x)}{\Omega}
\left[
\begin{pmatrix}
e & -e\\
-e & e
\end{pmatrix}
+
\begin{pmatrix}
1-e & 0\\
0 & 0
\end{pmatrix}
\right] \notag\\[4pt]
&=
\frac{f_{\mathrm{pred}}(x)}{\Omega}
\begin{pmatrix}
1 & -e\\
-e & e
\end{pmatrix}. \label{eq:appBcov_apred_B_final}
\end{align}
Equivalently,
\begin{equation}
\label{eq:appBcov_apred_B_final_explicit}
a_{\mathrm{pred}}^{(B)}(x)
=
\frac{1}{\Omega}\,\frac{mNP}{1+N}
\begin{pmatrix}
1 & -e\\
-e & e
\end{pmatrix}.
\end{equation}
Hence, the off-diagonal entry is
\begin{equation}
\label{eq:appBcov_B_offdiag}
\bigl(a_{\mathrm{pred}}^{(B)}(x)\bigr)_{12}
=
-\frac{1}{\Omega}\,e\,\frac{mNP}{1+N}.
\end{equation}
In particular, for \(x\in D\) and \(e\in(0,1]\),
\begin{equation}
\label{eq:appBcov_B_offdiag_negative}
\bigl(a_{\mathrm{pred}}^{(B)}(x)\bigr)_{12}<0.
\end{equation}

\subsection{Predation covariance under the effective coupled \texorpdfstring{\((-1,e)^\top\)}{(-1,e) transpose} closure}
\label{app:effective_covariance_derivation}

For the effective diffusion-level coupled closure, the predation channel is represented by
\[
\nu_{\mathrm{eff}}=\binom{-1}{e},
\qquad
f_{\mathrm{eff}}(x)=f_{\mathrm{pred}}(x)=\frac{mNP}{1+N}.
\]
Its predation covariance contribution is
\begin{equation}
\label{eq:appBcov_apred_eff_start}
a_{\mathrm{pred}}^{(\mathrm{eff})}(x)
=
\frac{1}{\Omega}\,f_{\mathrm{pred}}(x)\,\nu_{\mathrm{eff}}\nu_{\mathrm{eff}}^\top.
\end{equation}
Computing the outer product,
\begin{equation}
\label{eq:appBcov_nueff_outer}
\nu_{\mathrm{eff}}\nu_{\mathrm{eff}}^\top
=
\binom{-1}{e}\begin{pmatrix}-1 & e\end{pmatrix}
=
\begin{pmatrix}
1 & -e\\
-e & e^2
\end{pmatrix}.
\end{equation}
Substituting \eqref{eq:appBcov_nueff_outer} into \eqref{eq:appBcov_apred_eff_start}, we obtain
\begin{equation}
\label{eq:appBcov_apred_eff_final}
a_{\mathrm{pred}}^{(\mathrm{eff})}(x)
=
\frac{1}{\Omega}\,\frac{mNP}{1+N}
\begin{pmatrix}
1 & -e\\
-e & e^2
\end{pmatrix}.
\end{equation}
Therefore, the off-diagonal entry is
\begin{equation}
\label{eq:appBcov_eff_offdiag}
\bigl(a_{\mathrm{pred}}^{(\mathrm{eff})}(x)\bigr)_{12}
=
-\frac{1}{\Omega}\,e\,\frac{mNP}{1+N},
\end{equation}
which is strictly negative on \(D\) for \(e\in(0,1]\).

\subsection{Predation covariance under the split-channel comparator}
\label{app:split_covariance_derivation}

For the split-channel comparator, predation-related prey removal and predator birth are modeled as independent channels:
\[
\nu_4^{(S)}=\binom{-1}{0},
\qquad
\nu_5^{(S)}=\binom{0}{1},
\]
with intensities
\[
f_4^{(S)}(x)=f_{\mathrm{pred}}(x),
\qquad
f_5^{(S)}(x)=e\,f_{\mathrm{pred}}(x).
\]
The predation covariance contribution is
\begin{equation}
\label{eq:appBcov_apred_split_start}
a_{\mathrm{pred}}^{(S)}(x)
=
\frac{1}{\Omega}\left[
f_4^{(S)}(x)\,\nu_4^{(S)}(\nu_4^{(S)})^\top
+
f_5^{(S)}(x)\,\nu_5^{(S)}(\nu_5^{(S)})^\top
\right].
\end{equation}
Now
\begin{align}
\nu_4^{(S)}(\nu_4^{(S)})^\top
&=
\binom{-1}{0}\begin{pmatrix}-1 & 0\end{pmatrix}
=
\begin{pmatrix}
1 & 0\\
0 & 0
\end{pmatrix}, \label{eq:appBcov_nu4s_outer}\\[4pt]
\nu_5^{(S)}(\nu_5^{(S)})^\top
&=
\binom{0}{1}\begin{pmatrix}0 & 1\end{pmatrix}
=
\begin{pmatrix}
0 & 0\\
0 & 1
\end{pmatrix}. \label{eq:appBcov_nu5s_outer}
\end{align}
Hence
\begin{align}
a_{\mathrm{pred}}^{(S)}(x)
&=
\frac{1}{\Omega}\left[
f_{\mathrm{pred}}(x)
\begin{pmatrix}
1 & 0\\
0 & 0
\end{pmatrix}
+
e f_{\mathrm{pred}}(x)
\begin{pmatrix}
0 & 0\\
0 & 1
\end{pmatrix}
\right] \notag\\[4pt]
&=
\frac{1}{\Omega}\,\frac{mNP}{1+N}
\begin{pmatrix}
1 & 0\\
0 & e
\end{pmatrix}. \label{eq:appBcov_apred_split_final}
\end{align}
Therefore,
\begin{equation}
\label{eq:appBcov_split_offdiag}
\bigl(a_{\mathrm{pred}}^{(S)}(x)\bigr)_{12}=0.
\end{equation}
This confirms that the split-channel model is drift-compatible with the coupled closures (after parameter matching) but removes the same-event prey--predator covariance at the predation level.

\subsection{A general proposition: drift equivalence does not imply covariance equivalence}
\label{app:drift_vs_cov_general}

We now record an abstract statement (valid in an arbitrary dimension) that formalizes the distinction used repeatedly in Section~\ref{sec:stochastic_model}.

\begin{proposition}[Drift equivalence does not imply covariance equivalence]
\label{prop:drift_not_cov_general}
Fix \(d\in\mathbb N\), \(\Omega>0\), and let \(x\) be a state parameter. Consider two-channel representations
\[
\{(\nu_k,f_k(x))\}_{k=1}^{K}
\quad\text{and}\quad
\{(\tilde\nu_\ell,\tilde f_\ell(x))\}_{\ell=1}^{\tilde K},
\]
with \(\nu_k,\tilde\nu_\ell\in\mathbb R^d\) and nonnegative intensities \(f_k(x),\tilde f_\ell(x)\ge 0\). Define their diffusion drifts and covariances by
\begin{align}
b(x)&:=\sum_{k=1}^{K} f_k(x)\,\nu_k,
&
a(x)&:=\frac{1}{\Omega}\sum_{k=1}^{K} f_k(x)\,\nu_k\nu_k^\top, \label{eq:appBcov_ba_rep1}\\
\tilde b(x)&:=\sum_{\ell=1}^{\tilde K} \tilde f_\ell(x)\,\tilde\nu_\ell,
&
\tilde a(x)&:=\frac{1}{\Omega}\sum_{\ell=1}^{\tilde K} \tilde f_\ell(x)\,\tilde\nu_\ell\tilde\nu_\ell^\top. \label{eq:appBcov_ba_rep2}
\end{align}
Then \(b(x)=\tilde b(x)\) does not in general imply \(a(x)=\tilde a(x)\).

More precisely, equality of drifts imposes equality of first moments of the channel increments, while equality of covariances additionally requires equality of the corresponding second-moment tensors:
\[
\sum_{k=1}^{K} f_k(x)\,\nu_k=\sum_{\ell=1}^{\tilde K}\tilde f_\ell(x)\,\tilde\nu_\ell
\quad \centernot\implies \quad
\sum_{k=1}^{K} f_k(x)\,\nu_k\nu_k^\top=\sum_{\ell=1}^{\tilde K}\tilde f_\ell(x)\,\tilde\nu_\ell\tilde\nu_\ell^\top.
\]
\end{proposition}

\begin{proof}
The claim is immediate from the definitions \eqref{eq:appBcov_ba_rep1}--\eqref{eq:appBcov_ba_rep2}: drift depends on weighted first moments of channel increments, whereas covariance depends on weighted second moments.

To show strict non-implication, it suffices to provide a counterexample. In dimension \(d=2\), fix a scalar \(r>0\) and compare:
\[
\text{Representation I: }\quad (\nu,f)=\left(\binom{-1}{1},\,r\right),
\]
with
\[
\text{Representation II: }\quad 
(\tilde\nu_1,\tilde f_1)=\left(\binom{-1}{0},\,r\right),
\qquad
(\tilde\nu_2,\tilde f_2)=\left(\binom{0}{1},\,r\right).
\]
Then both representations produce the same drift,
\[
b=\tilde b=r\binom{-1}{1},
\]
but their covariance matrices differ:
\[
a=\frac{r}{\Omega}
\begin{pmatrix}
1 & -1\\
-1 & 1
\end{pmatrix},
\qquad
\tilde a=\frac{r}{\Omega}
\begin{pmatrix}
1 & 0\\
0 & 1
\end{pmatrix}.
\]
In particular, \(a_{12}=-r/\Omega\neq 0=\tilde a_{12}\), so \(a\neq \tilde a\).
\end{proof}

\begin{remark}[Relevance to the predation closures in this paper]
\label{rmk:appBcov_drift_vs_cov_RM}
Proposition~\ref{prop:drift_not_cov_general} is precisely the mechanism behind the distinction between the coupled predation closures (Bernoulli-coupled exact CTMC and effective \((-1,e)^\top\) diffusion closure) and the split-channel comparator. These constructions are drift-compatible after parameter matching, but their predation-related covariance contributions differ because their channel increment second moments differ.
\end{remark}

\subsection{Structural sign proposition for the predation-induced cross-covariance}
\label{app:structural_sign_proposition}

We now state and prove the structural sign result underlying the main text discussion.

\begin{proposition}[Structural negativity of the predation-induced prey--predator cross-covariance]
\label{prop:structural_negative_crosscov}
Let \(x=(N,P)\in D=(0,\infty)^2\), \(\Omega>0\), \(m>0\), and \(e\in(0,1]\). For the predation-related covariance contribution:
\begin{enumerate}[label=(\roman*),leftmargin=1.8em]
\item under the Bernoulli-coupled closure,
\[
\bigl(a_{\mathrm{pred}}^{(B)}(x)\bigr)_{12}
=
-\frac{1}{\Omega}\,e\,\frac{mNP}{1+N}<0;
\]
\item under the effective coupled \((-1,e)\) closure,
\[
\bigl(a_{\mathrm{pred}}^{(\mathrm{eff})}(x)\bigr)_{12}
=
-\frac{1}{\Omega}\,e\,\frac{mNP}{1+N}<0;
\]
\item under the split-channel comparator,
\[
\bigl(a_{\mathrm{pred}}^{(S)}(x)\bigr)_{12}=0.
\]
\end{enumerate}
Consequently, the negative predation-induced prey--predator cross-covariance is a structural signature of same-event prey-loss/predator-gain coupling, and is absent in the split-channel representation that decouples those increments.
\end{proposition}

\begin{proof}
Parts (i)--(iii) are immediate from the explicit formulas
\eqref{eq:appBcov_apred_B_final_explicit},
\eqref{eq:appBcov_apred_eff_final}, and
\eqref{eq:appBcov_apred_split_final}, respectively. Since \(x\in D\) implies \(N>0\) and \(P>0\), and since \(m>0\), \(e\in(0,1]\), and \(\Omega>0\), the quantity
\[
\frac{1}{\Omega}\,e\,\frac{mNP}{1+N}
\]
is strictly positive; therefore, the off-diagonal terms in (i) and (ii) are strictly negative.
\end{proof}

\begin{corollary}[Cross-covariance agreement of Bernoulli-coupled and effective coupled closures]
\label{cor:crosscov_B_equals_eff}
Under the assumptions of Proposition~\ref{prop:structural_negative_crosscov},
\[
\bigl(a_{\mathrm{pred}}^{(B)}(x)\bigr)_{12}
=
\bigl(a_{\mathrm{pred}}^{(\mathrm{eff})}(x)\bigr)_{12}
=
-\frac{1}{\Omega}\,e\,\frac{mNP}{1+N}.
\]
Thus, the effective \((-1,e)^\top\) closure preserves the predation-induced cross-covariance exactly, even though it is not an exact integer-valued CTMC jump representation when \(e\notin\mathbb N\).
\end{corollary}

\begin{proof}
Compare the off-diagonal entries in \eqref{eq:appBcov_apred_B_final_explicit} and \eqref{eq:appBcov_apred_eff_final}.
\end{proof}

\begin{remark}[Predator-variance difference between Bernoulli-coupled and effective coupled closures]
\label{rmk:appBcov_variance_difference}
Although Corollary~\ref{cor:crosscov_B_equals_eff} shows exact agreement of the predation-induced cross-covariance, the predator-variance entries differ:
\[
\bigl(a_{\mathrm{pred}}^{(B)}(x)\bigr)_{22}
=
\frac{1}{\Omega}\,e\,f_{\mathrm{pred}}(x),
\qquad
\bigl(a_{\mathrm{pred}}^{(\mathrm{eff})}(x)\bigr)_{22}
=
\frac{1}{\Omega}\,e^2\,f_{\mathrm{pred}}(x).
\]
This is the variance-level distinction noted in the main text: the effective \((-1,e)^\top\) closure preserves the coupled fluctuation direction and cross term, but not the exact Bernoulli offspring variance.
\end{remark}

\subsection{Matrix-level decomposition viewpoint}
\label{app:matrix_decomp_view}

For later reference, it is sometimes convenient to decompose the full covariance into non-predation and predation contributions:
\[
a(x)=a_{\mathrm{base}}(x)+a_{\mathrm{pred}}(x),
\]
where \(a_{\mathrm{base}}(x)\) collects prey birth, prey competition death, and predator death channels, and \(a_{\mathrm{pred}}(x)\) is one of
\[
a_{\mathrm{pred}}^{(B)}(x),\qquad
a_{\mathrm{pred}}^{(\mathrm{eff})}(x),\qquad
a_{\mathrm{pred}}^{(S)}(x).
\]
Since the base channels in the present construction contribute only diagonal terms, the sign and magnitude of the full off-diagonal covariance \(a_{12}(x)\) are determined entirely by the predation component. Therefore, Proposition~\ref{prop:structural_negative_crosscov} directly controls the sign of the full prey--predator cross-covariance in the coupled closures.

\end{appendices}

\bibliography{reference}

@article{kuehn_mathematical_2011,
	title = {A mathematical framework for critical transitions: {Bifurcations}, fast–slow systems and stochastic dynamics},
	volume = {240},
	doi = {10.1016/j.physd.2011.02.012},
	number = {12},
	journal = {Physica D: Nonlinear Phenomena},
	author = {Kuehn, C.},
	year = {2011},
	pages = {1020--1035},
}

@article{riebesell_paradox_1974,
	title = {Paradox of enrichment in competitive systems},
	volume = {55},
	copyright = {http://doi.wiley.com/10.1002/tdm\_license\_1.1},
	doi = {10.2307/1934634},
	language = {en},
	number = {1},
	journal = {Ecology},
	author = {Riebesell, J. F.},
	year = {1974},
	pages = {183--187},
}

@article{gilpin_enriched_1972,
	title = {Enriched predator–prey systems: {Theoretical} stability},
	volume = {177},
	doi = {10.1126/science.177.4052.902},
	number = {4052},
	journal = {Science},
	author = {Gilpin, M. E.},
	year = {1972},
	pages = {902--904},
}

@article{gillespie_general_1976,
	title = {A general method for numerically simulating the stochastic time evolution of coupled chemical reactions},
	volume = {22},
	doi = {10.1016/0021-9991(76)90041-3},
	number = {4},
	journal = {Journal of Computational Physics},
	author = {Gillespie, D. T.},
	year = {1976},
	pages = {403--434},
}

@article{gillespie_approximate_2001,
	title = {Approximate accelerated stochastic simulation of chemically reacting systems},
	volume = {115},
	doi = {10.1063/1.1378322},
	number = {4},
	journal = {The Journal of Chemical Physics},
	author = {Gillespie, D. T.},
	year = {2001},
	pages = {1716--1733},
}

@article{pineda-krch_tale_2007,
	title = {A tale of two cycles—distinguishing quasi-cycles and limit cycles in finite predator–prey populations},
	volume = {116},
	doi = {10.1111/j.2006.0030-1299.14940.x},
	number = {1},
	journal = {Oikos},
	author = {Pineda-Krch, M. and Blok, H. J. and Dieckmann, U. and Doebeli, M.},
	year = {2007},
	pages = {53--64},
}

@article{grima_effective_2010,
	title = {An effective rate equation approach to reaction kinetics in small volumes: {Theory} and application to biochemical reactions in nonequilibrium steady-state conditions},
	volume = {133},
	doi = {10.1063/1.3454685},
	number = {3},
	journal = {The Journal of Chemical Physics},
	author = {Grima, R.},
	year = {2010},
	pages = {035101},
}

@article{grima_construction_2011,
	title = {Construction and accuracy of partial differential equation approximations to the chemical master equation},
	volume = {84},
	doi = {10.1103/PhysRevE.84.056109},
	number = {5},
	journal = {Physical Review E},
	author = {Grima, R.},
	year = {2011},
	pages = {056109},
}

@article{schnoerr_approximation_2017,
	title = {Approximation and inference methods for stochastic biochemical kinetics—a tutorial review},
	volume = {50},
	doi = {10.1088/1751-8121/aa54d9},
	number = {9},
	journal = {Journal of Physics A: Mathematical and Theoretical},
	author = {Schnoerr, D. and Sanguinetti, G. and Grima, R.},
	year = {2017},
	pages = {093001},
}

@article{milshtein_first_1995,
	title = {A first approximation of the quasipotential in problems of the stability of systems with random non-degenerate perturbations},
	volume = {59},
	doi = {10.1016/0021-8928(95)00006-B},
	number = {1},
	journal = {Journal of Applied Mathematics and Mechanics},
	author = {Mil'shtein, G. N. and Ryashko, L. B.},
	year = {1995},
	pages = {47--56},
}

@article{bashkirtseva_stochastic_2004,
	title = {Stochastic sensitivity of {3D}-cycles},
	volume = {66},
	doi = {10.1016/j.matcom.2004.02.021},
	number = {1},
	journal = {Mathematics and Computers in Simulation},
	author = {Bashkirtseva, I. and Ryashko, L.},
	year = {2004},
	pages = {55--67},
}

@article{scheffer_anticipating_2012,
	title = {Anticipating critical transitions},
	volume = {338},
	doi = {10.1126/science.1225244},
	number = {6105},
	journal = {Science},
	author = {Scheffer, M. and Carpenter, S. R. and Lenton, T. M. and Bascompte, J. and Brock, W. and Dakos, V.},
	year = {2012},
	pages = {344--348},
}

@article{wissel_universal_1984,
	title = {A universal law of the characteristic return time near thresholds},
	volume = {65},
	doi = {10.1007/BF00384470},
	number = {1},
	journal = {Oecologia},
	author = {Wissel, C.},
	year = {1984},
	pages = {101--107},
}

@article{roughgarden_simple_1975,
	title = {A simple model for population dynamics in stochastic environments},
	volume = {109},
	doi = {10.1086/283039},
	number = {970},
	journal = {The American Naturalist},
	author = {Roughgarden, J.},
	year = {1975},
	pages = {713--736},
}

@article{hauzy_confronting_2013,
	title = {Confronting the paradox of enrichment to the metacommunity perspective},
	volume = {8},
	doi = {10.1371/journal.pone.0082969},
	language = {en},
	number = {12},
	journal = {PLoS ONE},
	author = {Hauzy, C. and Nadin, G. and Canard, E. and Gounand, I. and Mouquet, N.},
	year = {2013},
	pages = {e82969},
}

@article{jansen_regulation_1995,
	title = {Regulation of predator-prey systems through spatial interactions: a possible solution to the paradox of enrichment},
	volume = {74},
	shorttitle = {Regulation of predator-prey systems through spatial interactions},
	doi = {10.2307/3545983},
	number = {3},
	journal = {Oikos},
	author = {Jansen, V. A. A.},
	year = {1995},
	pages = {384},
}

@article{lin_bifurcation_2023,
	title = {Bifurcation and overexploitation in {Rosenzweig}-{MacArthur} model},
	volume = {28},
	doi = {10.3934/dcdsb.2022094},
	number = {1},
	journal = {Discrete and Continuous Dynamical Systems - B},
	author = {Lin, X. and Xu, Y. and Gao, D. and Fan, G.},
	year = {2023},
	pages = {690},
}

@book{freidlin_random_1998,
	address = {New York, NY},
	series = {Grundlehren der mathematischen {Wissenschaften}},
	title = {Random perturbations of dynamical systems},
	volume = {260},
	copyright = {http://www.springer.com/tdm},
	isbn = {978-1-4612-6839-0 978-1-4612-0611-8},
	doi = {10.1007/978-1-4612-0611-8},
	publisher = {Springer New York},
	author = {Freidlin, M. I. and Wentzell, A. D.},
	year = {1998},
}

@article{tian_additive_2017,
	title = {Additive noise driven phase transitions in a predator-prey system},
	volume = {46},
	doi = {10.1016/j.apm.2017.01.087},
	language = {en},
	journal = {Applied Mathematical Modelling},
	author = {Tian, C. and Lin, L. and Zhang, L.},
	year = {2017},
	pages = {423--432},
}

@article{wang_time-delay-induced_2018,
	title = {Time-delay-induced dynamical behaviors for an ecological vegetation growth system driven by cross-correlated multiplicative and additive noises},
	volume = {41},
	doi = {10.1140/epje/i2018-11668-9},
	language = {en},
	number = {5},
	journal = {The European Physical Journal E},
	author = {Wang, K.-K. and Ye, H. and Wang, Y.-J. and Li, S.-H.},
	year = {2018},
	pages = {60},
}

@article{wu_stochastic_2019,
	title = {Stochastic sensitivity analysis of noise-induced transitions in a predator-prey model with environmental toxins},
	volume = {16},
	doi = {10.3934/mbe.2019104},
	language = {en},
	number = {4},
	journal = {Mathematical Biosciences and Engineering},
	author = {Wu, D. and Wang, H. and Yuan, S.},
	year = {2019},
	pages = {2141--2153},
}

@article{pal_nonequilibrium_2025,
	title = {Nonequilibrium dynamics in a noise-induced predator–prey model},
	volume = {191},
	doi = {10.1016/j.chaos.2024.115884},
	language = {en},
	journal = {Chaos, Solitons \& Fractals},
	author = {Pal, S. and Banerjee, M. and Melnik, R.},
	year = {2025},
	pages = {115884},
}

@article{qi_threshold_2021,
	title = {Threshold behavior of a stochastic predator–prey system with prey refuge and fear effect},
	volume = {113},
	doi = {10.1016/j.aml.2020.106846},
	language = {en},
	journal = {Applied Mathematics Letters},
	author = {Qi, H. and Meng, X.},
	year = {2021},
	pages = {106846},
}

@article{liu_bidirectional_2025,
	title = {Bidirectional endothelial feedback drives turing-vascular patterning and drug-resistance niches: a hybrid {PDE}-agent-based study},
	volume = {12},
	shorttitle = {Bidirectional endothelial feedback drives turing-vascular patterning and drug-resistance niches},
	doi = {10.3390/bioengineering12101097},
	language = {en},
	number = {10},
	journal = {Bioengineering},
	author = {Liu, Z. and Wang, L. S. and Yu, J. and Zhang, J. and Martel, E. and Li, S.},
	year = {2025},
	pages = {1097},
}

@article{liang_global_2025,
	title = {Global well-posedness and stability of nonlocal damage-structured lineage model with feedback and dedifferentiation},
	volume = {13},
	doi = {10.3390/math13223583},
	language = {en},
	number = {22},
	journal = {Mathematics},
	author = {Liang, Y. and Wang, L. S. and Yu, J. and Liu, Z.},
	year = {2025},
	pages = {3583},
}

@article{wang_damage-structured_2026,
	title = {A damage-structured {PDE} model of stem cell hierarchies: {The} dual role of dedifferentiation in tissue homeostasis and aging},
	volume = {21},
	shorttitle = {A damage-structured {PDE} model of stem cell hierarchies},
	doi = {10.1371/journal.pone.0335163},
	language = {en},
	number = {2},
	journal = {PLOS One},
	author = {Wang, L. S. and Yu, J. and Liu, Z.},
	year = {2026},
	pages = {e0335163},
}

@misc{yu_beyond_2026,
	title = {Beyond diagonal noise: a better predator-prey modeling framework with cross-covariance},
	copyright = {Creative Commons Attribution 4.0 International},
	shorttitle = {Beyond diagonal noise},
	doi = {10.48550/ARXIV.2602.22489},
	publisher = {arXiv},
	author = {Yu, J. and Wang, L. S.},
	year = {2026},
	note = {Version Number: 1},
}

@misc{yu_full-covariance_2026,
	title = {Full-covariance chemical langevin predator--prey diffusion with absorbing boundaries},
	copyright = {Creative Commons Attribution 4.0 International},
	doi = {10.48550/ARXIV.2602.05336},
	publisher = {arXiv},
	author = {Yu, J. and Wang, L. S. and Gao, Y. and Liang, Y.},
	year = {2026},
	note = {Version Number: 1},
}

@misc{yu_chemotactic_2026,
	title = {Chemotactic feedback controls patterning in hybrid tumor--stroma model},
	copyright = {Creative Commons Attribution 4.0 International},
	doi = {10.48550/ARXIV.2601.16337},
	publisher = {arXiv},
	author = {Yu, J. and Wang, L. S. and Liu, Z. and Liu, J.},
	year = {2026},
	note = {Version Number: 1},
}

@article{may_limit_1972,
	title = {Limit cycles in predator-prey communities},
	volume = {177},
	doi = {10.1126/science.177.4052.900},
	language = {en},
	number = {4052},
	journal = {Science},
	author = {May, R. M.},
	year = {1972},
	pages = {900--902},
}

@article{abrams_invulnerable_1996,
	title = {Invulnerable prey and the paradox of enrichment},
	volume = {77},
	copyright = {http://onlinelibrary.wiley.com/termsAndConditions\#vor},
	doi = {10.2307/2265581},
	language = {en},
	number = {4},
	journal = {Ecology},
	author = {Abrams, P. A. and Walters, C. J.},
	year = {1996},
	pages = {1125--1133},
}

@article{kirk_enrichment_1998,
	title = {Enrichment can stabilize population dynamics: autotoxins and density dependence},
	volume = {79},
	copyright = {http://doi.wiley.com/10.1002/tdm\_license\_1.1},
	shorttitle = {Enrichment can stabilize population dynamics},
	doi = {10.1890/0012-9658(1998)079[2456:ECSPDA]2.0.CO;2},
	language = {en},
	number = {7},
	journal = {Ecology},
	author = {Kirk, K. L.},
	year = {1998},
	pages = {2456--2462},
}

@article{roy_stability_2007,
	title = {The stability of ecosystems: a brief overview of the paradox of enrichment},
	volume = {32},
	copyright = {http://www.springer.com/tdm},
	shorttitle = {The stability of ecosystems},
	doi = {10.1007/s12038-007-0040-1},
	language = {en},
	number = {2},
	journal = {Journal of Biosciences},
	author = {Roy, S. and Chattopadhyay, J.},
	year = {2007},
	pages = {421--428},
}

@article{mckane_predator-prey_2005,
	title = {Predator-prey cycles from resonant amplification of demographic stochasticity},
	volume = {94},
	copyright = {http://link.aps.org/licenses/aps-default-license},
	doi = {10.1103/PhysRevLett.94.218102},
	language = {en},
	number = {21},
	journal = {Physical Review Letters},
	author = {McKane, A. J. and Newman, T. J.},
	year = {2005},
	pages = {218102},
}

@article{boland_how_2008,
	title = {How limit cycles and quasi-cycles are related in systems with intrinsic noise},
	volume = {2008},
	doi = {10.1088/1742-5468/2008/09/P09001},
	number = {09},
	journal = {Journal of Statistical Mechanics: Theory and Experiment},
	author = {Boland, R. P. and Galla, T. and McKane, A. J.},
	year = {2008},
	pages = {P09001},
}

@article{scheffer_early-warning_2009,
	title = {Early-warning signals for critical transitions},
	volume = {461},
	copyright = {https://www.springer.com/tdm},
	doi = {10.1038/nature08227},
	language = {en},
	number = {7260},
	journal = {Nature},
	author = {Scheffer, M. and Bascompte, J. and Brock, W. A. and Brovkin, V. and Carpenter, S. R. and Dakos, V.},
	year = {2009},
	pages = {53--59},
}

@article{dakos_methods_2012,
	title = {Methods for detecting early warnings of critical transitions in time series illustrated using simulated ecological data},
	volume = {7},
	doi = {10.1371/journal.pone.0041010},
	language = {en},
	number = {7},
	journal = {PLoS ONE},
	author = {Dakos, V. and Carpenter, S. R. and Brock, W. A. and Ellison, A. M. and Guttal, V. and Ives, A. R.},
	year = {2012},
	pages = {e41010},
}

@article{boettiger_quantifying_2012,
	title = {Quantifying limits to detection of early warning for critical transitions},
	volume = {9},
	doi = {10.1098/rsif.2012.0125},
	language = {en},
	number = {75},
	journal = {Journal of The Royal Society Interface},
	author = {Boettiger, C. and Hastings, A.},
	year = {2012},
	pages = {2527--2539},
}

@article{boettiger_patterns_2013,
	title = {From patterns to predictions},
	volume = {493},
	doi = {10.1038/493157a},
	language = {en},
	number = {7431},
	journal = {Nature},
	author = {Boettiger, C. and Hastings, A.},
	year = {2013},
	pages = {157--158},
}

@article{lenton_tipping_2008,
	title = {Tipping elements in the {Earth}'s climate system},
	volume = {105},
	doi = {10.1073/pnas.0705414105},
	language = {en},
	number = {6},
	journal = {Proceedings of the National Academy of Sciences},
	author = {Lenton, T. M. and Held, H. and Kriegler, E. and Hall, J. W. and Lucht, W. and Rahmstorf, S.},
	year = {2008},
	pages = {1786--1793},
}

@book{van_kampen_stochastic_2011,
	address = {Amsterdam},
	edition = {3rd ed},
	series = {North-{Holland} {Personal} {Library}},
	title = {Stochastic processes in physics and chemistry},
	isbn = {978-0-08-047536-3 978-0-444-52965-7},
	language = {eng},
	publisher = {Elsevier Science \& Technology},
	author = {van Kampen, N. G.},
	year = {2011},
}

@article{kurtz_solutions_1970,
	title = {Solutions of ordinary differential equations as limits of pure jump markov processes},
	volume = {7},
	copyright = {https://www.cambridge.org/core/terms},
	doi = {10.2307/3212147},
	language = {en},
	number = {1},
	journal = {Journal of Applied Probability},
	author = {Kurtz, T. G.},
	year = {1970},
	pages = {49--58},
}

@article{kurtz_limit_1971,
	title = {Limit theorems for sequences of jump {Markov} processes approximating ordinary differential processes},
	volume = {8},
	copyright = {https://www.cambridge.org/core/terms},
	doi = {10.2307/3211904},
	language = {en},
	number = {2},
	journal = {Journal of Applied Probability},
	author = {Kurtz, T. G.},
	year = {1971},
	pages = {344--356},
}

@book{gardiner_stochastic_2009,
	address = {Berlin Heidelberg},
	edition = {4th ed},
	series = {Springer series in synergetics},
	title = {Stochastic methods: a handbook for the natural and social sciences},
	isbn = {978-3-642-08962-6 978-3-540-70712-7},
	shorttitle = {Stochastic methods},
	language = {eng},
	number = {13},
	publisher = {Springer},
	author = {Gardiner, C. W.},
	year = {2009},
}

@book{risken_fokker-planck_1996,
	address = {Berlin, Heidelberg},
	series = {Springer {Series} in {Synergetics}},
	title = {The {Fokker}-{Planck} equation: methods of solution and applications},
	volume = {18},
	copyright = {https://www.springernature.com/gp/researchers/text-and-data-mining},
	isbn = {978-3-540-61530-9 978-3-642-61544-3},
	shorttitle = {The fokker-planck equation},
	doi = {10.1007/978-3-642-61544-3},
	language = {en},
	publisher = {Springer Berlin Heidelberg},
	author = {Risken, H.},
	year = {1996},
}

@article{lugo_quasicycles_2008,
	title = {Quasicycles in a spatial predator-prey model},
	volume = {78},
	copyright = {http://link.aps.org/licenses/aps-default-license},
	doi = {10.1103/PhysRevE.78.051911},
	language = {en},
	number = {5},
	journal = {Physical Review E},
	author = {Lugo, C. A. and McKane, A. J.},
	year = {2008},
	pages = {051911},
}

@article{black_stochastic_2012,
	title = {Stochastic formulation of ecological models and their applications},
	volume = {27},
	copyright = {https://www.elsevier.com/tdm/userlicense/1.0/},
	doi = {10.1016/j.tree.2012.01.014},
	language = {en},
	number = {6},
	journal = {Trends in Ecology \& Evolution},
	author = {Black, A. J. and McKane, A. J.},
	year = {2012},
	pages = {337--345},
}

@article{gillespie_exact_1977,
	title = {Exact stochastic simulation of coupled chemical reactions},
	volume = {81},
	doi = {10.1021/j100540a008},
	language = {en},
	number = {25},
	journal = {The Journal of Physical Chemistry},
	author = {Gillespie, D. T.},
	year = {1977},
	pages = {2340--2361},
}

@article{gillespie_chemical_2000,
	title = {The chemical {Langevin} equation},
	volume = {113},
	doi = {10.1063/1.481811},
	language = {en},
	number = {1},
	journal = {The Journal of Chemical Physics},
	author = {Gillespie, D. T.},
	year = {2000},
	pages = {297--306},
}

@book{ethier_markov_1986,
	address = {New York},
	edition = {1},
	series = {Wiley {Series} in {Probability} and {Statistics}},
	title = {Markov processes: characterization and convergence},
	copyright = {http://doi.wiley.com/10.1002/tdm\_license\_1.1},
	isbn = {978-0-471-08186-9 978-0-470-31665-8},
	shorttitle = {Markov processes},
	doi = {10.1002/9780470316658},
	language = {en},
	publisher = {Wiley},
	author = {Ethier, S. N. and Kurtz, T. G.},
	year = {1986},
}

@book{anderson_stochastic_2015,
	address = {Cham},
	title = {Stochastic analysis of biochemical systems},
	copyright = {https://www.springernature.com/gp/researchers/text-and-data-mining},
	isbn = {978-3-319-16894-4 978-3-319-16895-1},
	doi = {10.1007/978-3-319-16895-1},
	language = {en},
	publisher = {Springer International Publishing},
	author = {Anderson, D. F. and Kurtz, T. G.},
	year = {2015},
}

@article{wang_algebraicspectral_2026,
	title = {Algebraic–spectral thresholds and discrete–continuous stability transfer in {Leslie}–{Gower} systems},
	volume = {34},
	doi = {10.3934/era.2026013},
	number = {1},
	journal = {Electronic Research Archive},
	author = {Wang, L. S. and Yu, J.},
	year = {2026},
	pages = {251--290},
}

@article{alonso_stochastic_2007,
	title = {Stochastic amplification in epidemics},
	volume = {4},
	copyright = {https://royalsociety.org/journals/ethics-policies/data-sharing-mining/},
	doi = {10.1098/rsif.2006.0192},
	language = {en},
	number = {14},
	journal = {Journal of The Royal Society Interface},
	author = {Alonso, D. and McKane, A. J. and Pascual, M.},
	year = {2007},
	pages = {575--582},
}

@article{wang_analysis_2025,
	title = {Analysis framework for stochastic predator–prey model with demographic noise},
	volume = {27},
	doi = {10.1016/j.rinam.2025.100621},
	language = {en},
	journal = {Results in Applied Mathematics},
	author = {Wang, L. S. and Yu, J.},
	year = {2025},
	pages = {100621},
}

@article{wang_analysis_2025-1,
	title = {Analysis and mean-field limit of a hybrid {PDE}-{ABM} modeling angiogenesis-regulated resistance evolution},
	volume = {13},
	doi = {10.3390/math13172898},
	language = {en},
	number = {17},
	journal = {Mathematics},
	author = {Wang, L. S. and Yu, J. and Li, S. and Liu, Z.},
	year = {2025},
	pages = {2898},
}

@article{chatterjee_predatorprey_2024,
	title = {A predator–prey model with prey refuge: under a stochastic and deterministic environment},
	volume = {112},
	shorttitle = {A predator–prey model with prey refuge},
	doi = {10.1007/s11071-024-09756-9},
	language = {en},
	number = {15},
	journal = {Nonlinear Dynamics},
	author = {Chatterjee, A. and Abbasi, M. A. and Venturino, E. and Zhen, J. and Haque, M.},
	year = {2024},
	pages = {13667--13693},
}

@book{kuznetsov_elements_2023,
	address = {Cham},
	series = {Applied {Mathematical} {Sciences}},
	title = {Elements of applied bifurcation theory},
	volume = {112},
	copyright = {https://www.springernature.com/gp/researchers/text-and-data-mining},
	isbn = {978-3-031-22006-7 978-3-031-22007-4},
	doi = {10.1007/978-3-031-22007-4},
	language = {en},
	publisher = {Springer International Publishing},
	author = {Kuznetsov, Y. A.},
	year = {2023},
}

@book{bazykin_nonlinear_1998,
	address = {Singapore},
	series = {World {Scientific} {Series} on {Nonlinear} {Science} {Series} {A}},
	title = {Nonlinear dynamics of interacting populations},
	volume = {11},
	isbn = {978-981-02-1685-6 978-981-279-872-5},
	doi = {10.1142/2284},
	language = {en},
	publisher = {WORLD SCIENTIFIC},
	author = {Bazykin, A. D. and Khibnik, A. I and Krauskopf, B.},
	year = {1998},
}

@article{feller_grundlagen_1939,
	title = {Die grundlagen der volterraschen theorie des kampfes ums dasein in wahrscheinlichkeitstheoretischer behandlung},
	volume = {5},
	copyright = {http://www.springer.com/tdm},
	doi = {10.1007/BF01602932},
	language = {de},
	number = {1},
	journal = {Acta Biotheoretica},
	author = {Feller, W.},
	year = {1939},
	pages = {11--40},
}

@book{gillespie_markov_1992,
	address = {Boston},
	title = {Markov processes: an introduction for physical scientists},
	isbn = {978-0-08-091837-2},
	shorttitle = {Markov processes},
	language = {eng},
	publisher = {Academic Press},
	author = {Gillespie, D. T.},
	year = {1992},
}

@incollection{anderson_continuous_2011,
	address = {New York, NY},
	title = {Continuous time markov chain models for chemical reaction networks},
	isbn = {978-1-4419-6765-7 978-1-4419-6766-4},
	doi = {10.1007/978-1-4419-6766-4_1},
	language = {en},
	booktitle = {Design and {Analysis} of {Biomolecular} {Circuits}},
	publisher = {Springer New York},
	author = {Anderson, D. F. and Kurtz, T. G.},
	year = {2011},
	pages = {3--42},
}

@article{kurtz_strong_1978,
	title = {Strong approximation theorems for density dependent {Markov} chains},
	volume = {6},
	copyright = {https://www.elsevier.com/tdm/userlicense/1.0/},
	doi = {10.1016/0304-4149(78)90020-0},
	language = {en},
	number = {3},
	journal = {Stochastic Processes and their Applications},
	author = {Kurtz, T. G.},
	year = {1978},
	pages = {223--240},
}

@article{thomas_rigorous_2012,
	title = {Rigorous elimination of fast stochastic variables from the linear noise approximation using projection operators},
	volume = {86},
	copyright = {http://link.aps.org/licenses/aps-default-license},
	doi = {10.1103/PhysRevE.86.041110},
	language = {en},
	number = {4},
	journal = {Physical Review E},
	author = {Thomas, P. and Grima, R. and Straube, A. V.},
	year = {2012},
	pages = {041110},
}

@article{cao_efficient_2006,
	title = {Efficient step size selection for the tau-leaping simulation method},
	volume = {124},
	doi = {10.1063/1.2159468},
	language = {en},
	number = {4},
	journal = {The Journal of Chemical Physics},
	author = {Cao, Y. and Gillespie, D. T. and Petzold, L. R.},
	year = {2006},
	pages = {044109},
}

@book{nisbet_modelling_1982,
	address = {Chichester ; New York},
	title = {Modelling fluctuating populations},
	publisher = {Wiley},
	author = {Nisbet, R. M. and Gurney, W. S. C.},
	year = {1982},
}

@misc{yu_mode-wise_2026,
	title = {Mode-wise spectral criteria for coupled mass transport in hybrid {PDE}--{ODE} tumor microenvironments},
	copyright = {Creative Commons Attribution 4.0 International},
	doi = {10.48550/ARXIV.2601.20204},
	publisher = {arXiv},
	author = {Yu, J. and Wang, L. S. and Liu, Z. and Liu, J.},
	year = {2026},
	note = {Version Number: 1},
}

@article{barraquand_moving_2017,
	title = {Moving forward in circles: challenges and opportunities in modelling population cycles},
	volume = {20},
	shorttitle = {Moving forward in circles},
	doi = {10.1111/ele.12789},
	language = {en},
	number = {8},
	journal = {Ecology Letters},
	author = {Barraquand, F. and Louca, S. and Abbott, K. C. and Cobbold, C. A. and Cordoleani, F. and DeAngelis, D. L.},
	year = {2017},
	pages = {1074--1092},
}

@article{mckane_stochastic_2014,
	title = {Stochastic pattern formation and spontaneous polarisation: the linear noise approximation and beyond},
	volume = {76},
	copyright = {http://www.springer.com/tdm},
	shorttitle = {Stochastic pattern formation and spontaneous polarisation},
	doi = {10.1007/s11538-013-9827-4},
	language = {en},
	number = {4},
	journal = {Bulletin of Mathematical Biology},
	author = {McKane, A. J. and Biancalani, T. and Rogers, T.},
	year = {2014},
	pages = {895--921},
}

@article{butler_robust_2009,
	title = {Robust ecological pattern formation induced by demographic noise},
	volume = {80},
	copyright = {http://link.aps.org/licenses/aps-default-license},
	doi = {10.1103/PhysRevE.80.030902},
	language = {en},
	number = {3},
	journal = {Physical Review E},
	author = {Butler, T. and Goldenfeld, N.},
	year = {2009},
	pages = {030902},
}

@article{biancalani_stochastic_2010,
	title = {Stochastic {Turing} patterns in the {Brusselator} model},
	volume = {81},
	copyright = {http://link.aps.org/licenses/aps-default-license},
	doi = {10.1103/PhysRevE.81.046215},
	language = {en},
	number = {4},
	journal = {Physical Review E},
	author = {Biancalani, T. and Fanelli, D. and Di Patti, F.},
	year = {2010},
	pages = {046215},
}

@article{biancalani_stochastic_2011,
	title = {Stochastic waves in a {Brusselator} model with nonlocal interaction},
	volume = {84},
	copyright = {http://link.aps.org/licenses/aps-default-license},
	doi = {10.1103/PhysRevE.84.026201},
	language = {en},
	number = {2},
	journal = {Physical Review E},
	author = {Biancalani, T. and Galla, T. and McKane, A. J.},
	year = {2011},
	pages = {026201},
}

@article{rozhnova_stochastic_2010,
	title = {Stochastic effects in a seasonally forced epidemic model},
	volume = {82},
	copyright = {http://link.aps.org/licenses/aps-default-license},
	doi = {10.1103/PhysRevE.82.041906},
	language = {en},
	number = {4},
	journal = {Physical Review E},
	author = {Rozhnova, G. and Nunes, A.},
	year = {2010},
	pages = {041906},
}

@article{black_stochastic_2010,
	title = {Stochastic amplification in an epidemic model with seasonal forcing},
	volume = {267},
	copyright = {https://www.elsevier.com/tdm/userlicense/1.0/},
	doi = {10.1016/j.jtbi.2010.08.014},
	language = {en},
	number = {1},
	journal = {Journal of Theoretical Biology},
	author = {Black, A. J. and McKane, A. J.},
	year = {2010},
	pages = {85--94},
}

@article{boland_limit_2009,
	title = {Limit cycles, complex {Floquet} multipliers, and intrinsic noise},
	volume = {79},
	copyright = {http://link.aps.org/licenses/aps-default-license},
	doi = {10.1103/PhysRevE.79.051131},
	language = {en},
	number = {5},
	journal = {Physical Review E},
	author = {Boland, R. P. and Galla, T. and McKane, A. J.},
	year = {2009},
	pages = {051131},
}

@book{allen_introduction_2010,
	address = {New York},
	edition = {2},
	title = {An introduction to stochastic processes with applications to biology},
	isbn = {978-0-429-18460-4},
	doi = {10.1201/b12537},
	language = {en},
	publisher = {Chapman and Hall/CRC},
	author = {Allen, L. J. S.},
	year = {2010},
}

@book{dembo_large_2010,
	address = {Berlin, Heidelberg},
	series = {Stochastic {Modelling} and {Applied} {Probability}},
	title = {Large deviations techniques and applications},
	volume = {38},
	copyright = {http://www.springer.com/tdm},
	isbn = {978-3-642-03310-0 978-3-642-03311-7},
	doi = {10.1007/978-3-642-03311-7},
	publisher = {Springer Berlin Heidelberg},
	author = {Dembo, A. and Zeitouni, O.},
	year = {2010},
}

@article{bashkirtseva_sensitivity_2000,
	title = {Sensitivity analysis of the stochastically and periodically forced {Brusselator}},
	volume = {278},
	copyright = {https://www.elsevier.com/tdm/userlicense/1.0/},
	doi = {10.1016/S0378-4371(99)00453-7},
	language = {en},
	number = {1-2},
	journal = {Physica A: Statistical Mechanics and its Applications},
	author = {Bashkirtseva, I. A and Ryashko, L. B},
	year = {2000},
	pages = {126--139},
}

@article{bashkirtseva_sensitivity_2011,
	title = {Sensitivity analysis of stochastic attractors and noise-induced transitions for population model with {Allee} effect},
	volume = {21},
	issn = {1054-1500, 1089-7682},
	doi = {10.1063/1.3647316},
	language = {en},
	number = {4},
	journal = {Chaos: An Interdisciplinary Journal of Nonlinear Science},
	author = {Bashkirtseva, I. and Ryashko, L.},
	year = {2011},
	pages = {047514},
}

@article{ryashko_sensitivity_2018,
	title = {Sensitivity analysis of the noise-induced oscillatory multistability in {Higgins} model of glycolysis},
	volume = {28},
	doi = {10.1063/1.4989982},
	language = {en},
	number = {3},
	journal = {Chaos: An Interdisciplinary Journal of Nonlinear Science},
	author = {Ryashko, L.},
	year = {2018},
	pages = {033602},
}

@article{alexandrov_noise-induced_2018,
	title = {Noise-induced transitions and shifts in a climate–vegetation feedback model},
	volume = {5},
	doi = {10.1098/rsos.171531},
	language = {en},
	number = {4},
	journal = {Royal Society Open Science},
	author = {Alexandrov, D. V. and Bashkirtseva, I. A. and Ryashko, L. B.},
	year = {2018},
	pages = {171531},
}

@article{bashkirtseva_noise-induced_2016,
	title = {Noise-induced extinction in {Bazykin}-{Berezovskaya} population model},
	volume = {89},
	doi = {10.1140/epjb/e2016-70345-6},
	language = {en},
	number = {7},
	journal = {The European Physical Journal B},
	author = {Bashkirtseva, I. and Ryashko, L.},
	year = {2016},
	pages = {165},
}

@article{bashkirtseva_sensitivity_2005,
	title = {Sensitivity and chaos control for the forced nonlinear oscillations},
	volume = {26},
	copyright = {https://www.elsevier.com/tdm/userlicense/1.0/},
	doi = {10.1016/j.chaos.2005.03.029},
	language = {en},
	number = {5},
	journal = {Chaos, Solitons \& Fractals},
	author = {Bashkirtseva, I. and Ryashko, L.},
	year = {2005},
	pages = {1437--1451},
}

@article{dakos_slowing_2008,
	title = {Slowing down as an early warning signal for abrupt climate change},
	volume = {105},
	doi = {10.1073/pnas.0802430105},
	language = {en},
	number = {38},
	journal = {Proceedings of the National Academy of Sciences},
	author = {Dakos, V. and Scheffer, M. and Van Nes, E. H. and Brovkin, V. and Petoukhov, V. and Held, H.},
	year = {2008},
	pages = {14308--14312},
}

@article{held_detection_2004,
	title = {Detection of climate system bifurcations by degenerate fingerprinting},
	volume = {31},
	copyright = {http://onlinelibrary.wiley.com/termsAndConditions\#vor},
	doi = {10.1029/2004GL020972},
	language = {en},
	number = {23},
	journal = {Geophysical Research Letters},
	author = {Held, H. and Kleinen, T.},
	year = {2004},
	pages = {2004GL020972},
}

@article{lade_early_2012,
	title = {Early warning signals for critical transitions: a generalized modeling approach},
	volume = {8},
	shorttitle = {Early warning signals for critical transitions},
	doi = {10.1371/journal.pcbi.1002360},
	language = {en},
	number = {2},
	journal = {PLoS Computational Biology},
	author = {Lade, S. J. and Gross, T.},
	editor = {Pascual, M.},
	year = {2012},
	pages = {e1002360},
}

@article{dutta_robustness_2018,
	title = {Robustness of early warning signals for catastrophic and non‐catastrophic transitions},
	volume = {127},
	doi = {10.1111/oik.05172},
	language = {en},
	number = {9},
	journal = {Oikos},
	author = {Dutta, P. S. and Sharma, Y. and Abbott, K. C.},
	year = {2018},
	pages = {1251--1263},
}

@article{dakos_resilience_2015,
	title = {Resilience indicators: prospects and limitations for early warnings of regime shifts},
	volume = {370},
	shorttitle = {Resilience indicators},
	doi = {10.1098/rstb.2013.0263},
	language = {en},
	number = {1659},
	journal = {Philosophical Transactions of the Royal Society B: Biological Sciences},
	author = {Dakos, V. and Carpenter, S. R. and Van Nes, E. H. and Scheffer, M.},
	year = {2015},
	pages = {20130263},
}

@article{ditlevsen_tipping_2010,
	title = {Tipping points: {Early} warning and wishful thinking},
	volume = {37},
	copyright = {http://onlinelibrary.wiley.com/termsAndConditions\#vor},
	shorttitle = {Tipping points},
	doi = {10.1029/2010GL044486},
	language = {en},
	number = {19},
	journal = {Geophysical Research Letters},
	author = {Ditlevsen, P. D. and Johnsen, S. J.},
	year = {2010},
	pages = {2010GL044486},
}

@article{lenton_early_2011,
	title = {Early warning of climate tipping points},
	volume = {1},
	copyright = {http://www.springer.com/tdm},
	doi = {10.1038/nclimate1143},
	language = {en},
	number = {4},
	journal = {Nature Climate Change},
	author = {Lenton, T. M.},
	year = {2011},
	pages = {201--209},
}

@article{chen_eigenvalues_2019,
	title = {Eigenvalues of the covariance matrix as early warning signals for critical transitions in ecological systems},
	volume = {9},
	doi = {10.1038/s41598-019-38961-5},
	language = {en},
	number = {1},
	journal = {Scientific Reports},
	author = {Chen, S. and ODea, E. B. and Drake, J. M. and Epureanu, B. I.},
	year = {2019},
	pages = {2572},
}

@article{oregan_theory_2013,
	title = {Theory of early warning signals of disease emergence and leading indicators of elimination},
	volume = {6},
	copyright = {http://creativecommons.org/licenses/by/2.0},
	doi = {10.1007/s12080-013-0185-5},
	language = {en},
	number = {3},
	journal = {Theoretical Ecology},
	author = {O’Regan, S. M. and Drake, J. M.},
	year = {2013},
	pages = {333--357},
}

@article{krkosek_signals_2014,
	title = {On signals of phase transitions in salmon population dynamics},
	volume = {281},
	doi = {10.1098/rspb.2013.3221},
	language = {en},
	number = {1784},
	journal = {Proceedings of the Royal Society B: Biological Sciences},
	author = {Krkošek, M. and Drake, J. M.},
	year = {2014},
	pages = {20133221},
}

@article{gillespie_stochastic_2007,
	title = {Stochastic simulation of chemical kinetics},
	volume = {58},
	doi = {10.1146/annurev.physchem.58.032806.104637},
	language = {en},
	number = {1},
	journal = {Annual Review of Physical Chemistry},
	author = {Gillespie, D. T.},
	year = {2007},
	pages = {35--55},
}

@book{bartlett_stochastic_1970,
	address = {London},
	edition = {1960th},
	series = {Methuen's monographs on applied probability and statistics},
	title = {Stochastic population models in ecology and epidemiology},
	isbn = {978-0-416-52330-0},
	language = {eng},
	publisher = {Methuen},
	author = {Bartlett, M. S.},
	year = {1970},
}

@article{mcquarrie_stochastic_1967,
	title = {Stochastic approach to chemical kinetics},
	volume = {4},
	copyright = {https://www.cambridge.org/core/terms},
	doi = {10.2307/3212214},
	language = {en},
	number = {3},
	journal = {Journal of Applied Probability},
	author = {McQuarrie, D. A.},
	year = {1967},
	pages = {413--478},
}

@article{rosenzweig_graphical_1963,
	title = {Graphical representation and stability conditions of predator-prey interactions},
	volume = {97},
	doi = {10.1086/282272},
	language = {en},
	number = {895},
	journal = {The American Naturalist},
	author = {Rosenzweig, M. L. and MacArthur, R. H.},
	year = {1963},
	pages = {209--223},
}

@article{macarthur_species_1970,
	title = {Species packing and competitive equilibrium for many species},
	volume = {1},
	copyright = {https://www.elsevier.com/tdm/userlicense/1.0/},
	doi = {10.1016/0040-5809(70)90039-0},
	language = {en},
	number = {1},
	journal = {Theoretical Population Biology},
	author = {MacArthur, R.},
	year = {1970},
	pages = {1--11},
}

@article{rosenzweig_paradox_1971,
	title = {Paradox of enrichment: destabilization of exploitation ecosystems in ecological time},
	volume = {171},
	shorttitle = {Paradox of enrichment},
	doi = {10.1126/science.171.3969.385},
	language = {en},
	number = {3969},
	journal = {Science},
	author = {Rosenzweig, M. L.},
	year = {1971},
	pages = {385--387},
}

@article{bashkirtseva_stochastic_2013,
	title = {Stochastic sensitivity analysis of the noise-induced excitability in a model of a hair bundle},
	volume = {87},
	copyright = {http://link.aps.org/licenses/aps-default-license},
	doi = {10.1103/PhysRevE.87.052711},
	language = {en},
	number = {5},
	journal = {Physical Review E},
	author = {Bashkirtseva, I. and Neiman, A. B. and Ryashko, L.},
	year = {2013},
	pages = {052711},
}

\end{document}